\documentclass[11pt,reqno,oneside]{amsart}

\usepackage[width=36pc]{geometry}

\usepackage{amssymb}
\usepackage{amsthm}
\usepackage{amsmath}
\usepackage{upgreek}
\usepackage{dsfont}
\usepackage{mathrsfs} 
\usepackage[all]{xy}

\newtheorem{thm}{Theorem}[section]
\newtheorem{lem}[thm]{Lemma}
\newtheorem{prop}[thm]{Proposition}
\newtheorem{cor}[thm]{Corollary}

\theoremstyle{definition}
\newtheorem{dfn}[thm]{Definition}
\newtheorem{ex}[thm]{Example}

\theoremstyle{remark}
\newtheorem{remark}[thm]{Remark}
\newtheorem{remarks}[thm]{Remarks}

\newtheorem{notation}[thm]{Notation}

\newcommand{\CA}{{\mathcal{A}}}
\newcommand{\CD}{{\mathcal{D}}}
\newcommand{\CE}{{\mathcal{E}}}
\newcommand{\BE}{{\mathbb{E}}}

\newcommand{\CI}{{\mathcal{I}}}
\newcommand{\CJ}{{\mathcal{J}}}
\newcommand{\CK}{{\mathcal{K}}}
\newcommand{\CL}{{\mathcal{L}}}
\newcommand{\CB}{{\mathcal{B}}}

\newcommand{\CO}{{\mathcal{O}}}

\newcommand{\CR}{{\mathcal{R}}}

\newcommand{\af}{\alpha}
\newcommand{\bt}{\beta}
\newcommand{\gm}{\gamma}
\newcommand{\dt}{\delta}
\newcommand{\ep}{\epsilon}

\newcommand{\ld}{\lambda}

\newcommand{\sm}{\sigma}

\newcommand{\om}{\omega}

\newcommand{\Om}{\Omega}

\let\oldbibitem\bibitem% Copy \bibitem into \oldbibitem
\renewcommand{\bibitem}[2][]{\oldbibitem{#2}}% 

\begin{document}

%\numberwithin{equation}{section}

\title{AF-embeddable labeled graph $C^*$-algebras}

\author[J. A. Jeong]{Ja A Jeong$^{\dagger}$}
\thanks{Research partially supported by NRF-2015R1C1A2A01052516 and 
2018R1D1A1B07041172$^{\dagger}$}
\address{
Department of Mathematical Sciences and Research Institute of Mathematics\\
Seoul National University\\
Seoul, 08826\\
Korea} \email{jajeong\-@\-snu.\-ac.\-kr }

\author[G. H. Park]{Gi Hyun Park$^{\ddagger}$}
\thanks{Research partially supported by Hanshin University$^{\ddagger}$}
\address{
Department of Financial Mathematics\\
Hanshin University\\
Osan, 18101\\
Korea} \email{ghpark\-@\-hs.\-ac.\-kr }

\keywords{ AF-embeddability, labeled graph $C^*$-algebra, quasidiagonal $C^*$-algebra}

\subjclass[2010]{37A55, 46L05, 46L55}

\begin{abstract} 
Finiteness conditions for $C^*$-algebras like AF-embeddability, 
quasidiagonality, stable finiteness have been studied by many authors 
and shown to be equivalent 
for certain classes of $C^*$-algebras. 
For example, Schfhauser proves that these conditions are all equivalent for 
$C^*$-algebras of compact topological graphs, and similar results 
were established by Clark, an Huef, and Sims for 
$k$-graph algebras.  
If $C^*(E,\mathcal L)$ is a labeled graph $C^*$-algebra over finite alphabet, 
it can be viewed as   a $C^*$-algebra of a compact topological graph. 
For these labeled graph $C^*$-algebras, we  
provide conditions on labeled paths and show that they are 
equivalent to AF-embeddability of $C^*(E,\mathcal L)$. 
\end{abstract}

\maketitle

\section{Introduction}

\noindent
In this paper we are concerned with 
AF-embeddability, quasidiagonality, and stable finiteness 
of separable $C^*$-algebras associated to labeled graphs. 
If a $C^*$-algebra $A$ can be embedded into an AF algebra, 
one can show that it has a faithful representation $\pi$ of $A$ such that 
$\pi(A)$ consists of quasidiagonal operators, namely 
there is an approximately central sequence of finite rank projections 
which converges strongly to the unit. 
A $C^*$-algebra with this property is called quasidiagonal. 
Quasidiagonality is known to have connection with nuclearity of $C^*$-algebras 
\cite{Ha:1987}, and 
it played an important role in Elliott's programme to 
classify simple nuclear $C^*$-algebras 
(for example, see \cite{TWW:2017} among many others). 
For quasidiagonal $C^*$-algebras, we refer the reader to 
Brown's expository paper \cite{Br:2004}, and there we particularly notice 
that quasidiagonal $C^*$-algebras are stably finite.

For certain classes of $C^*$-algebras, 
AF-embeddability and stable finiteness are known to be 
equivalent, hence they are 
equivalent to quasidiagonality as well. 
Among those $C^*$-algebras are 
the crossed products considered by Pimsner and 
Brown, respectively: 

\vskip 1pc

\begin{thm}\label{Pimsner thm} {\rm (\cite[Theorem 9]{Pim:1983})} 
\label{Pimsner thm} 
Let $X$ be a compact metrizable space and $\sm$ be a homeomorphism of $X$. 
Then the following are equivalent: 
\begin{enumerate} 
\item[$(a)$] $C(X)\times_\sm \mathbb Z$ is AF embeddable,
\item[$(b)$] $C(X)\times_\sm \mathbb Z$ is quasidiagonal, 
\item[$(c)$] $C(X)\times_\sm \mathbb Z$ is stably finite,
\item[$(d)$] $C(X)\times_\sm \mathbb Z$ is finite,
\item[$(e)$] every point $x\in X$ is {\it pseudoperiodic} for $\sm$ 
in the sense that for each $x_1\in X$ and $\varepsilon>0$, 
there exist $x_2, \dots, x_n$ in $X$ such that 
$\rho(\sm(x_i),x_{i+1})<\varepsilon$ for all $1\leq i\leq n$, where 
the subscripts are taken modulo $n$ and $\rho$ is a metric compatible with 
the topology of $X$.
\end{enumerate}
\end{thm} 

\vskip 1pc 
   
\begin{thm}\label{Brown thm} {\rm (\cite[Theorem 0.2]{Br:1998})} 
If $\af$ is an automorphism on an AF algebra $A$, then 
the following are equivalent: 
\begin{enumerate} 
\item[{\rm (a)}] $A\times_\af \mathbb Z$ is AF embeddable,
\item[{\rm (b)}] $A\times_\af \mathbb Z$ is quasidiagonal, 
\item[{\rm (c)}] $A\times_\af \mathbb Z$ is stably finite,
\item[{\rm (d)}] $H_\af\cap K_0^+(A)=\{0\}$, 
where $H_\af:=\{ \af_*(x)-x\mid x\in K_0(A)\}$.
\end{enumerate}
\end{thm} 

\vskip 1pc 
  
Besides, Schafhauser \cite{Sch:2015-1} 
proved that these properties are all equivalent 
for graph $C^*$-algebras $C^*(E)$ and that 
$C^*(E)$ is finite exactly when $E$ has no loops with an exit.  
In \cite{CHS:2016},  Clark, an Huef, and Sims considered 
these properties with 
$k$-graph $C^*$-algebras and proved that quasidiagonality and stable finiteness 
are equivalent.  
They also provided two more conditions equivalent to quasidiagonality, 
and 
obtained that they are all equivalent to AF-embeddability for 2-graph algebras. 

Motivated by these works, 
in this paper we study
AF-embeddability, quasidiagonality and stable finiteness of 
labeled graph $C^*$-algebras with focus on finding 
equivalent conditions in terms of labeled graph or its path spaces.
We first show the following:

\vskip 1pc 

\begin{prop}\label{intro prop-crossed product} 
{\rm(Proposition \ref{prop-crossed product})}
Let $(E,\CL)$ be a weakly left-resolving, set-finite labeled graph. 
Then the following are equivalent: 
\begin{enumerate} 
\item[$(a)$] $C^*(E,\CL)$ is AF embeddable,
\item[$(b)$] $C^*(E,\CL)$ is quasidiagonal, 
\item[$(c)$] $C^*(E,\CL)$ is stably finite.
%\item[{\rm (d)}] $C^*(E,\CL)$ is  finite.
\end{enumerate}
\end{prop}

\vskip 1pc 

\noindent 
Actually this comes easily from the structural result 
of labeled graph $C^*$-algebras $C^*(E,\CL)$ by Bates, Pask, and Willis 
\cite{BPW:2012}
that  $C^*(E,\CL)$ is stably isomorphic to the crossed product 
of an AF algebra by the integers because once this is observed, then 
we can apply Brown's  result Theorem \ref{Brown thm}. 

On the other hand, 
if we restrict ourselves to labeled graphs $(E,\CL)$ 
where  the labeling map 
$\CL:E^1\to \CA$ is onto a finite alphabet $\CA$, 
then we can say more, 
especially on path space conditions 
as stated in $(e)$ of the following theorem:

\vskip 1pc

\begin{thm} {\rm (Theorem \ref{thm main})} \label{intro thm 1} 
Let $E$ have no sinks or sources 
and $(E,\CL)$ be a labeled graph over finite alphabet $\CA$. 
Then the following are equivalent: 
\begin{enumerate} 
\item[$(a)$] $C^*(E,\CL)$  is AF-embeddable,
\item[$(b)$] $C^*(E,\CL)$  is quasidiagonal,
\item[$(c)$] $C^*(E,\CL)$  is stably finite,
\item[$(d)$] $C^*(E,\CL)$  is finite,
\item[$(e)$] $(E,\CL)$ is pseudo-periodic, and for $\af,\bt\in \CL(E^k)$, $k\geq 1$,   
$$r(\af)\cap r(\bt)\neq \emptyset \iff \af=\bt.$$ 
\end{enumerate}
\end{thm} 
 
\vskip 1pc
 
\noindent
To prove Theorem \ref{intro thm 1}, we first view  
our labeled graph $C^*$-algebra $C^*(E,\CL)$ as 
the $C^*$-algebra of a Boolean dynamical system investigated by 
Carlsen, Ortega, and Pardo \cite{COP:2017} where they 
proved that a $C^*$-algebra of a Boolean dynamical system is 
isomorphic to a $C^*$-algebra of a topological graph which was 
introduced and  studied intensively by Katsura 
in \cite{Ka:2004, Ka:2006-1, Ka:2006-2, Ka:2008}.
Thus if $(E,\CL)$ is a labeled graph considered in 
Theorem \ref{intro thm 1}, then there always exists a topological graph 
$\BE=(\BE^0,\BE^1,d,r)$ such that $C^*(\BE)\cong C^*(E,\CL)$. 
Then we can possibly apply Schfhauser's results 
\cite{Sch:2015} where he proved that 
if $\BE$ is a compact topological graph, then 
$C^*(\BE)\cong C(\BE^\infty)\times_\sm \mathbb Z$ and 
so could apply  Theorem \ref{Pimsner thm} to obtain that 
AF-embeddability, quasidiagonalty, and stable finiteness 
are all equivalent to finiteness for $C^*(\BE)$. 
Thus equivalence of four properties 
$(a)-(d)$ of Theorem \ref{intro thm 1} follows from the 
isomorphism $C^*(\BE)\cong C^*(E,\CL)$. 

For condition $(e)$ of Theorem \ref{intro thm 1} 
which we are most interested in, we 
look at a combinatorial characterization 
on $\BE$ obtained by Schfhauser for a finite $C^*$-algebra $C^*(\BE)$.
(This characterization was obtained again from Theorem \ref{intro thm 1}). 
In our case of labeled graph $C^*$-algebras 
$C^*(E,\CL)\cong C^*(\BE)$, 
the compact vertex space $\BE^0$ is the Stone's spectrum 
of the smallest accommodating set $\CE$ 
generated by the range sets of all labeled paths 
(thus consists of the ultrafilters of $\CE$). 
Hence, when we consider the infinite paths $\BE^\infty$  
as sequences of vertices (hence sequences of ultrafilters 
in $\CE$), it is not clear how to explain them by 
using the labeled paths in the labeled graph $(E,\CL)$. 
Thus, the equivalence of the condition  $(e)$, described in terms of 
labeled graph itself, 
to the rest of four finiteness conditions is where the 
present paper makes a contribution. 
 
The second condition about ranges of labeled paths in 
Theorem \ref{intro thm 1}.(e) is equivalent to 
the injectivity of $d:\BE^1\to \BE^0$  of a topological graph 
$(\BE^0,\BE^1,d,r)$ associated to $(E,\CL)$ as 
we will see in Proposition \ref{disjoint range}. 
Note that this injectivity condition of $d$ is necessary for $C^*(\BE)$  
to be finite (or AF-embeddable).  
Under this necessary condition, we obtain 
the following: 

\vskip 1pc 

\begin{thm} {\rm (Theorem \ref{thm simple})}\label{intro thm simple} 
Let $E$ have no sinks or sources 
and $(E,\CL)$ be a labeled graph over finite alphabet $\CA$ with an 
infinite set $\overline{\CL(E^{\,\infty}_{-\infty})}$. 
If $(\BE^0,\BE^1, d,r)$ is the topological graph of $(E,\CL)$  
such that  $d:\BE^1\to \BE^0$ is injective,  
then we have the following:
\begin{enumerate}
\item[{\rm (a)}] 
$C^*(E,\CL)$ is simple if and only if  
every infinite word ${\bf a}\in \overline{\CL(E^{\,\infty}_{-\infty})}$ 
contains every finite path  in $\CL^*(E)$ as its finite word.
 
\item[{\rm (b)}] 
If $C^*(E,\CL)$ is simple, it is always AF-embeddable.
\end{enumerate}
\end{thm}

\vskip 1pc

\noindent 
The space $\overline{\CL(E^{\,\infty}_{-\infty})}$ 
denotes the closure of the two-sided infinite paths of $(E,\CL)$
in the compact space $\CA^{\mathbb Z}$. 
Then it consists of all ${\bf a}\in \CA^{\mathbb Z}$ such that 
 each of its finite words must belong to $\CL^*(E)$ 
(that is, must appear a labeled path in $(E,\CL)$). 
In view of Theorem \ref{Pimsner thm}, we show in  Lemma \ref{lem cross product} 
that the $C^*$-algebra $C^*(E,\CL)$  is 
isomorphic to the crossed product 
$C(\overline{\CL(E^{\,\infty}_{-\infty})})\times_\tau \mathbb Z$ 
whenever $d$ is injective.
We can also state the following corollary for simple labeled graph $C^*$-algebras: 

\vskip 1pc

\begin{cor} {\rm (Corollary \ref{cor main})}
Let $E$ have no sinks or sources 
and $(E,\CL,\CE)$ be a labeled space over finite alphabet $\CA$. 
If $C^*(E,\CL,\CE)$ is simple, the following are equivalent: 
\begin{enumerate}
\item[{\rm (a)}] $C^*(E,\CL,\CE)$ is AF-embeddable.
\item[{\rm (b)}] $C^*(E,\CL,\CE)$ is quasidiagonal.
\item[{\rm (c)}] $C^*(E,\CL,\CE)$ is stably finite.
\item[{\rm (d)}] $C^*(E,\CL,\CE)$ is finite.
\item[{\rm (e)}] $r(a)\cap r(b)=\emptyset$ for $a,b\in \CA$ with $a\neq b$.
\item[{\rm (f)}] $C^*(E,\CL,\CE)^\gm$ is   
the diagonal subalgebra 
$\overline{\rm span}\{s_\af p_A s_\af^*\mid \af\in \CL^*(E),\ A\in \CE\}$.
\end{enumerate}
\end{cor} 

\vskip 1pc

This paper is organized as follows. 
In Section 2, we set up notation and briefly 
review some of the useful  
facts on labeled graph $C^*$-algebras,  topological 
graphs $\BE$ associated to labeled graph $(E,\CL)$, 
 and quasidiagonal (or AF-embeddable) $C^*$-algebras.  
Then in Section 3, it is obtained that 
AF-embeddability, quasidiagonality, and stable finiteness 
are all equivalent for labeled graph $C^*$-algebras. 
Section 4 is devoted to investigate the labeled graph $C^*$-algebras 
over finite alphabet. We analyze 
the topological graphs associated to  labeled graphs and 
  finiteness of their $C^*$-algebras 
  based on Schafhauser's result \cite{Sch:2015}. 
For this we employ various path spaces and fully make use of 
the generalized vertices to prove many lemmas that help us
explain the Stones' spectrum of the accommodating 
set $\CE$ and obtain the results that we believe are very  
convenient and tractable.

\vskip 1pc 
\noindent 
{\bf Acknowledgement} The first author is grateful to Astrid an Huef 
for many helpful suggestions.
 
\vskip 1pc 

\section{Preliminaries} 

\noindent
In this section we set up notation and review 
definitions and basic results we need in this paper. 
For more details, we refer the reader to \cite{BCP:2015}, 
\cite{COP:2017},  \cite{JKaP:2017}, and \cite{Sch:2015}.

\subsection{Labeled graphs and their $C^*$-algebras}
A  {\it directed graph} $E=(E^0,E^1,r,s)$
consists of  the vertex set $E^0$ and the edge set $E^1$ 
together with the range, source maps $r$, $s: E^1\to E^0$. 
We call  a vertex $v\in E^0$ a {\it sink} 
(a {\it source}, respectively)
if $s^{-1}(v)=\emptyset$  ($r^{-1}(v)=\emptyset$, respectively).
If every vertex in $E$ emits only finitely many edges, 
$E$ is called {\it row-finite}.  

For each $n\geq 1$, 
$E^n$  denotes the set of all  paths of length $n$, and 
the vertices in $E^0$ are regarded as finite paths of length zero. 
The maps $r,s$ naturally extend to the set 
$E^*=\cup_{n \geq 0} E^n$ of all finite paths, 
especially with $r(v)=s(v)=v$ for $v \in E^0$. 
By $E^{\infty}$, 
we denote the set of all  right-infinite paths 
$x=\ld_{1}\ld_{2}\cdots$, 
where $r(\ld_i)=s(\ld_{i+1})$ for all $i\geq 1$ 
and we define $s(x):= s(\ld_1)$. 
Similarly we will consider the set 
$$E_{-\infty}:=\{x=\cdots \ld_3\ld_2\ld_1\mid \ld_n\in E^1\ \text{and } 
r(\ld_{n+1})=s(\ld_n),\ n\geq 1\}$$
of all left-infinite paths($r(x):=r(\ld_1)$) 
and the set 
$$E^\infty_{-\infty}:=\{ x=\cdots \ld_{-1}\ld_0\ld_1\cdots\mid 
\ld_i\in E^1\ \text{and } r(\ld_{i})=s(\ld_{i+1}),\ i\in \mathbb Z\}$$ 
of all bi-infinite paths in $E$.

For $A,B\subset E^0$ and $n\geq 0$, we use the following notation 
 $$ AE^n: =\{\ld\in E^n :  s(\ld)\in A\},\ \
  E^nB: =\{\ld\in E^n : r(\ld)\in B\},$$
 and  $AE^nB: =AE^n\cap E^nB$ with  
$E^n v:=E^n\{v\}$, $vE^n:=\{v\}E^n$.
Also the sets of paths 
like $E^{\geq k}$, $AE^{\geq k}$, and $AE^\infty$ which 
have their obvious meaning will be used.  
A  {\it loop} is a finite path $\ld\in E^{\geq 1}$ 
such that $r(\ld)=s(\ld)$, and 
an {\it exit} of a loop $\ld$ is a path 
$\dt\in E^{\geq 1}$ such that  
$|\dt|\leq |\ld|,\ s(\dt)=s(\ld), \text{ and } 
\dt\neq \ld_1\cdots \ld_{|\dt|}.$ 
A graph $E$ is said to satisfy {\it Condition} (L)  
if every loop has an exit. 
% and $E$ is said to satisfy {\it condition} (K) 
%if no vertex in $E$ is the source vertex of exactly 
%one loop which does not return to its source vertex more than once.

A {\it labeled graph} $(E,\CL)$ over $\CA$ consists of 
a directed graph $E$ and  a {\it labeling map} 
$\CL:E^1\to \CA$ which is always assumed to be onto. 
Given a graph $E$, one can define a so-called 
{\it trivial labeling} map 
 $\CL_{id}:=id:E^1\to E^1$ which is the identity map 
 on $E^1$ with the alphabet $E^1$. 
The labeling map naturally extends to any 
finite and infinite labeled paths, namely 
if $\ld=\ld_1\cdots \ld_n\in E^n$, then  
 $\CL(\ld):=\CL(\ld_1)\cdots \CL(\ld_n)\in \CL(E^n)\subset \CA^*$, and  
 similarly to  infinite paths. 
 We often call these labeled paths just paths for convenience if 
 there is no risk of confusion,  
and use notation $\CL^*(E):=\CL(E^{\geq 1})$.
For a vertex $v\in E^0$ and a vertex subset $A\subset E^0$, 
we set $\CL(v): =v$ and $\CL(A): =A$, respectively. 
 A subpath  $\af_i\cdots \af_j$ of  
$\af=\af_1\af_2\cdots\af_{|\af|}\in \CL^*(E)$ 
 is denoted by
 $\af_{[i,j]}$ for $1\leq i\leq j\leq |\af|$.
 The range and source of a path $\af\in \CL^*(E)$ are defined 
to be the following sets of vertices
 \begin{align*}
r(\af) &=\{r(\ld) \in E^0 \,:\, \ld\in E^{\geq 1},\,\CL(\ld)=\af\},\\
 s(\af) &=\{s(\ld) \in E^0 \,:\, \ld\in E^{\geq 1},\, \CL(\ld)=\af\},
\end{align*}
and the {\it relative range of $\af\in \CL^*(E)$  
with respect to $A\subset  E^0$} is defined by
$$
 r(A,\af)=\{r(\ld)\,:\, \ld\in AE^{\geq 1},\ \CL(\ld)=\af \}.
$$
A collection  $\CB$ of subsets of $E^0$ is said to be
 {\it closed under relative ranges} for $(E,\CL)$ if 
$r(A,\af)\in \CB$ whenever 
 $A\in \CB$ and $\af\in \CL^*(E)$. 
We call $\CB$ an {\it accommodating set}~ for $(E,\CL)$
 if it is closed under relative ranges,
 finite intersections and unions and contains 
the ranges $r(\af)$ of all paths $\af\in \CL^*(E)$.
In other words, an accommodating set $\CB$ is a Boolean algebra such that  
$r(\af)\in \CB$ for all $\af\in \CL^*(E)$.

If $\CB$ is accommodating for $(E,\CL)$, 
the triple $(E,\CL,\CB)$ is called
 a {\it labeled space}. 
We say that a labeled space $(E,\CL,\CB)$ is {\it set-finite}
 ({\it receiver set-finite}, respectively) if 
for every $A\in \CB$ and $k\geq 1$ 
 the set  $\CL(AE^k)$ ($\CL(E^k A)$, respectively) is finite.
  A labeled space $(E,\CL,\CB)$ is said to be {\it weakly left-resolving} 
if 
 $$r(A,\af)\cap r(B,\af)=r(A\cap B,\af)$$
holds  for all $A,B\in \CB$ and  $\af\in \CL^*(E)$.
If $\CB$ is closed under relative complements, 
we call $(E,\CL, \CB)$ a {\it normal} labeled space.

\vskip 1pc 

\begin{notation}
For $A\in \CB$, we will use the following notation
$$\CI_A:=\{B\in \CB: B\subset A\}.$$ 
Note that $\CI_A$ is an ideal of the Boolean algebra $\CB$.
\end{notation}

\vskip 1pc 

\noindent 
{\bf Assumptions.}   
Throughout this paper, we assume that 
graphs $E$ have no sinks and sources,  
and labeled spaces $(E,\CL,\CB)$ are  
weakly left-resolving,  set-finite, receiver set-finite, and 
normal. 

\vskip 1pc 
 
\begin{dfn} 
\label{def-representation}
A {\it representation} of a labeled space $(E,\CL,\CB)$
is a family of projections $\{p_A\,:\, A\in \CB\}$ and
partial isometries
$\{s_a\,:\, a\in \CA\}$ such that for $A, B\in \CB$ and $a, b\in \CA$,
\begin{enumerate}
\item[(i)]  $p_{\emptyset}=0$, $p_{A\cap B}=p_Ap_B$, and
$p_{A\cup B}=p_A+p_B-p_{A\cap B}$,
\item[(ii)] $p_A s_a=s_a p_{r(A,a)}$,
\item[(iii)] $s_a^*s_a=p_{r(a)}$ and $s_a^* s_b=0$ unless $a=b$,
\item[(iv)]\label{CK4}  $p_A=\sum_{a\in \CL(AE^1)} s_a p_{r(A,a)}s_a^*.$ 
\end{enumerate}
\end{dfn}

\vskip 1pc  
\noindent
It is known \cite{BCP:2015,BP1:2007} that 
given a labeled space  $(E,\CL,\CB)$, 
there exists a $C^*$-algebra $C^*(E,\CL,\CB)$ generated by 
a universal representation $\{s_a,p_A\}$ of $(E,\CL,\CB)$, 
so that if $\{t_a, q_A\}$ is a representation of $(E,\CL,\CB)$ 
in a $C^*$-algebra $B$, there exists a $*$-homomorphism 
$$\phi: C^*(E,\CL,\CB)\to B$$ such that 
$\phi(s_a)=t_a$ and $\phi(p_A)=q_A$ for all $a\in \CA$ and 
$A\in \CB$.  
The $C^*$-algebra $C^*(E,\CL,\CB)$ is unique up to isomorphism,  
and we simply write $$C^*(E,\CL,\CB)=C^*(s_a,p_A)$$ 
to indicate the generators $s_a, p_A$ that 
are nonzero for all $a\in \CA$ and  $A\in \CB$, $A\neq \emptyset$. 
 
\vskip 1pc 
 
\begin{dfn} 
We call the $C^*$-algebra $C^*(E,\CL,\CB)$ generated by 
a universal representation of $(E,\CL,\CB)$ 
the {\it $C^*$-algebra of
a labeled space} $(E,\CL,\CB)$.
For a labeled graph $(E,\CL)$, there are many accommodating 
sets to be considered to form a labeled space. 
By $\CE$  we denote the smallest accommodating set 
for which $(E,\CL, \CE)$ is a normal labeled space. 
We  call $C^*(E,\CL,\CE)$ 
the {\it labeled graph $C^*$-algebra} of $(E,\CL)$ and often
denote it by $C^*(E,\CL)$.
\end{dfn}  

\vskip 1pc

Recall  
that a labeled space $(E,\CL, \CE)$ is said to be  {\it disagreeable}  
 if for any nonempty set $A\in \CE$ and a path $\bt\in \CL^*(E)$, 
there is an $n\geq 1$ such that $\CL(AE^{|\bt|n})\neq \{\bt^n\}$ 
(\cite[Definition 5.2]{BP2:2009} and \cite[Proposition 3.2]{JKaP:2017}). 
Below is the {\it Cuntz-Krieger uniqueness theorem} 
for  labeled graph $C^*$-algebras: 

\vskip 1pc 

\begin{thm} {\rm (\cite[Theorem 5.5]{BP1:2007}, \cite[Theorem 9.9]{COP:2017})} 
\label{CK uniqueness thm} 
Let  $\{t_a, q_A\}$ be a representation of a labeled space $(E,\CL,\CE)$ 
such that $q_A\neq 0$ for all nonempty $A\in \CE$. 
If $(E,\CL,\CE)$ is disagreeable, then 
the canonical homomorphism $\phi:C^*(E,\CL,\CE)=C^*(s_a, p_A)\to 
 C^*(t_a, q_A)$ such that $\phi(s_a)=t_a$ and 
$\phi(p_A)=q_A$  is an isomorphism. 
\end{thm} 
 
\vskip 1pc

\begin{remark}\label{basics} 
Let $(E,\CL,\CB)$ be a labeled space with $C^*(E,\CL,\CB)=C^*(s_a,p_A)$. 
By $\ep$, we denote a symbol (not in $\CL^*(E)$)
such that  
$a\epsilon=\epsilon a$, $r(\ep)=E^0$, and $r(A,\ep)=A$ 
for all $a\in \CA$ and $A \subset E^0$. 
Let $\CL^{\#}(E):=\CL^*(E)\cup \{\ep\}$ and 
let $s_\ep$ denote the unit of 
the multiplier algebra of $C^*(E,\CL,\CB)$.  
Then one can easily show that 
$$
C^*(E,\CL,\CB)=\overline{\rm span}\{s_\af p_A s_\bt^*\,:\,
\af,\,\bt\in  \CL^{\#}(E) ~\text{and}~ A \subseteq r(\af)\cap r(\bt)\}. 
$$
\end{remark}

\vskip 1pc

\subsection{Generalized vertices $[v]_l$}
For each $l\geq 1$,  
the relation $\sim_l$ on  $E^0$ given by 
$v\sim_l w$ if and only if 
$\CL(E^{\leq l} v)=\CL(E^{\leq l} w)$ 
is an equivalence relation, and the equivalence class 
$[v]_l$ of $v\in E^0$ is called a {\it generalized vertex} 
(or  a vertex simply).  
   If $k>l$,  then $[v]_k\subset [v]_l$ is obvious and
   $[v]_l=\cup_{i=1}^m [v_i]_{l+1}$
   for some vertices  $v_1, \dots, v_m\in [v]_l$ 
(\cite[Proposition 2.4]{BP2:2009}). 
Moreover,  we have
\begin{eqnarray}\label{acc set}
\CE=\big\{ \cup_{i=1}^n [v_i]_l:\, v_i\in E^0,\  l\geq 1,\,  n\geq 0 \big\},
\end{eqnarray}
with the convention  $\sum_{i=1}^0 [v_i]_l:=\emptyset$ 
by \cite[Proposition 2.3]{JKK:2014}. 
For each $l\geq 1$, we denote  by $\Om^l(E)$ 
the set of all generalized vertices $[v]_l$, $v\in E^0$.

$(E,\CL,\CE)$ is {\it strongly cofinal} if 
for each $x\in \overline{\CL(E^\infty)}$ and $[v]_l\in \CE$, 
there exist 
an $N\geq 1$ and a finite number of paths 
$\ld_1, \dots, \ld_m\in \CL(E^{\geq 1})$ such that 
$$r(x_{[1,N]})\subset \cup_{i=1}^m r([v]_l,\ld_i).$$  
It is now well known that $C^*(E,\CL,\CE)$ is simple if and only if 
$(E,\CL,\CE)$ is disagreeable and strongly cofinal 
(\cite{BP2:2009}, \cite{JP:2018}). 
  
\vskip 1pc 

\subsection{Infinite path spaces $\CL(E^\infty)$, $\CL(E_{-\infty})$, and 
$\CL(E^\infty_{-\infty})$} 
 
Let $\CA$ be a countable alphabet and let  
$\CA^*$ ($\CA^{\mathbb N}$ and
$\CA^{\mathbb Z}$, respectively) denote   
the set of all finite words 
(right-infinite and bi-infinite words respectively) in symbols of $\CA$. 
For $\af\in \CA^*$,  $(\af]$ denotes the set 
$\{\bt\af\mid \bt\in \CA^*\}\cup \{\af\}$ of all finite words 
that end with $\af$. Similarly 
$[\af)$ denotes the set of all finite words that start with $\af$.
With the product topology, 
$\CA^{\mathbb N}$ is a metrizable space in which the cylinder sets 
$$ [\af :=\{{\bf a}=a_1 a_2\cdots \in \CA^{\mathbb N} \mid 
{\bf a}_{[1,n]}:=a_1\cdots a_{n}=\af \},$$ 
$\af \in \CA^*$ ($|\af|=n$), form 
a countable basis of open-closed sets  
(see Section 7.2 of \cite{Ki:1998}). 
Similarly, the product space $\CA^{\mathbb Z}$ is 
also a metrizable space and the cylinder sets
$$ [\af.\bt] : =\{{\bf a}=\cdots  a_{-1}a_0a_1\cdots\mid 
{\bf a}_{[-n,-1]}=\af,\ {\bf a}_{[0,n]}=\bt\}, $$
$\af,\bt\in \CA^*$, form a basis, where 
${\bf a}_{[-n,-1]}:=a_{-n}\cdots a_{-1}$ and 
${\bf a}_{[0,n]}:=a_{0}\cdots a_{n}$.
We will also consider left-infinite words  
${\bf a}=\cdots a_{-3} a_{-2} a_{-1}$, 
and by $\CA^{-\mathbb N}$ we denote the set of all 
such words which also forms a metrizable space with a countable basis 
consisting of cylinder sets
$$ \af] :=\{{\bf a}=\cdots a_{-2} a_{-1}\in \CA^{-\mathbb N} \mid 
{\bf a}_{[-n,-1]}:=a_{-n}\cdots a_{-1}=\af \},$$ 
$\af \in \CA^*$ ($|\af|=n$).

Let $(E,\CL,\CE)$ be a labeled space with $\CL(E^1)=\CA$. 
By $\overline{\CL(E^\infty)}$, we denote  
the closure of  all infinite labeled paths $\CL(E^\infty)$ 
in $\CA^{\mathbb N}$. 
Then it is not hard to see that $\overline{\CL(E^\infty)}$ 
consists of all right-infinite words (or sequences) 
 ${\bf a} \in \CA^{\mathbb N}$ such that 
every finite word of ${\bf a}$ 
occurs as a labeled path in $\CL^*(E)$;
$$\overline{\CL(E^\infty)}:=\{{\bf a}\in \CA^{\mathbb N}\mid 
{\bf a}_{[1,n]}\in \CL(E^n) \ \text{for all } n\geq 1\}.$$  
Similarly, the closure $\overline{\CL(E_{-\infty})}$ of 
the left-infinite paths $\CL(E_{-\infty})$ in $\CA^{-\mathbb N}$ consists of 
the left-infinite words ${\bf a}\in \CA^{-\mathbb N}$ such that 
for each $n\geq 1$, ${\bf a}_{[-n,-1]}\in \CL(E^n)$.
The closure $\overline{\CL(E^{\infty}_{-\infty})}$ of the 
bi-infinite labeled paths in $\CA^{\mathbb Z}$ is of course 
consisting of bi-infinite words ${\bf a}\in \CA^{\mathbb Z}$ 
satisfying ${\bf a}_{[-n,n]}\in \CL(E^{2n+1})$ for each $n\geq 1$.

Note that if $\CA$ is finite, then $\CA^{\mathbb N}$, 
$\CA^{-\mathbb N}$, and $\CA^{\mathbb Z}$ are all compact spaces 
(with compact-open cylinder sets), hence our labeled path spaces 
$\overline{\CL(E^\infty)}$, $\overline{\CL(E_{-\infty})}$, and 
$\overline{\CL(E^\infty_{-\infty})}$ are all compact as well. 

\vskip 1pc

\subsection{Topological graphs of labeled spaces} 

A {\it topological graph} $\BE=(\BE^0,\BE^1, r,d)$ 
consists of locally compact second countable spaces $\BE^i$, $i=0,1$, and 
continuous maps $d,r:\BE^1\to \BE^0$ such that 
$d$ is a local homeomorphism.  
The $C^*$-algebra $C^*(\BE)$ of a topological graph $\BE$ was 
introduced and studied systematically in 
\cite{Ka:2004, Ka:2006-1, Ka:2006-2, Ka:2008}.   
Here we review from \cite[Section 5]{COP:2017} that 
there is a topological graph $\BE$ associated to a labeled space 
$(E,\CL,\CE)$ such that 
$C^*(\BE)\cong C^*(E,\CL,\CE)$ whenever  
$E$ has no sinks or sources and $\CA:=\CL(E^1)$ is a finite alphabet.
 
Let $E$ be a graph with no sinks or sources 
and $(E,\CL,\CE)$ be a labeled space over a finite alphabet $\CA=\CL(E^1)$. 
Then there exists a Boolean dynamical system 
$(\CE, \CA, \theta)$, $\theta=\{\theta_a\}_{a\in \CA})$, 
where $\theta_a$ has a compact range and a closed domain 
in the sense of \cite[Definition 3.3]{COP:2017}. 
To see this, first note that for each $a\in \CA$, 
the map  $\theta_a:\CE\to \CE$ given by 
$$\theta_a(A)=r(A,a)$$
is an action on  $\CE$, that is, 
$\theta_a$ is a Boolean algebra homomorphism such that 
$\theta_a(\emptyset)=\emptyset$.
Since $\theta_a(A)=r(A,a)\subset r(a)$  for $A\in \CE$  and  
 $$r(a)=\cup_{b\in \CA} r(ba)=\theta_a(\cup_{b\in \CA} r(b)),$$
 the range set $r(a)$ is the least upper-bound for $\{\theta_a(A)\}_{A\in \CE}$, 
 and  thus $\theta_a$ has  compact range $\CR_{\theta_a}:=r(a)$.  
Moreover, since $E^0=\cup_{b\in \CA} r(b)\in \CE$ ($E$ has no sources), 
we see from $\theta_a (E^0)=r(a)=\CR_{\theta_a}$ that 
 $\theta_a$ has a closed domain $\CD_{\theta_a}:=E^0$.

Recall that a {\it filter} of a Boolean algebra $\CB$ is 
a subset $\xi\subset \CB$ which satisfies the following properties: 
\begin{enumerate} 
\item[$\cdot$] $\emptyset \notin \xi$, 
\item[$\cdot$] if $A\in \xi$ and $A\subset B\in \CB$, then $B\in \xi$,
\item[$\cdot$] if $A,B\in \xi$, then $A\cap B\in \xi$. 
\end{enumerate} 
Moreover, a filter $\xi$ is called a {\it ultrafilter} if $A\in \xi$ and 
$A=B\cup C$ for $B,C\in \CB$, then either $B\in \xi$ or $C\in \xi$. 
An ultrafilter is a maximal filter.

Let $\widehat{\CE}$ be the set of all ultrafilters of $\CE$ 
with topology of which the cylinder sets 
$$Z(A):=\{\xi\in \widehat{\CE}: A\in \xi\},$$  
$A\in \CE$, form a basis. 
$\widehat{\CE}$ is called the Stone's spectrum of $\CE$. 
$Z(A)$ is compact and open for all $A\in \CE$, $A\neq \emptyset$. 
If $\CI$ is an ideal of $\CE$, it is a Boolean algebra itself and 
its Stone's spectrum  $\widehat{\CI}$ is 
canonically  embedded in $\widehat\CE$ via the map 
$\xi\mapsto \iota(\xi): \widehat{\CI}\to\widehat{\CE}$ given by 
$$\iota(\xi):=\{A\in \CE: A\supset B\ \text{ for } \exists B\in \xi\}.$$ 
In particular, for an ideal of the form 
$\CI_A:=\{B\in \CE: B\subset A\}$, $A\in \CE$, 
one can easily see that  
$$\iota(\widehat{\CI}_A):=\{ \iota(\xi): \xi\in \widehat{\CI}_A \}=Z(A),$$ 
hence $\iota(\widehat{\CI}_A)$ is a 
compact open subset of $\widehat{\CE}$ (see \cite[Section 2]{COP:2017}).
If $\CJ$ is an ideal of $\CE$ containing $\CI$, 
then 
$\iota(\xi)=\iota(\iota_{\widehat\CJ}(\xi))$ 
for each $\xi\in \widehat{\CI}$, where 
$\iota_{\widehat\CJ}(\xi):=
\{ A\in \CJ: A\supset B\ \text{ for } \exists B\in \xi  \}\in \widehat{\CJ}$.  
Particularly, for ideals $\CJ\supset \CI_A$ of $\CE$ we have  
\begin{eqnarray}\label{ideals}
\iota(\widehat{\CI}_A) =\iota(\iota_{\widehat\CJ}(\widehat{\CI}_A)) .
\end{eqnarray}
For each $a\in \CA$, we will write 
$\CI_a:=\CI_{r(a)}$, which then gives $\iota(\widehat{\CI}_a) =Z(r(a))$.

\vskip 1pc 

\begin{remark} 
Let $\CA=\CL(E^1)$ be finite.
For the Boolean homomorphism $\theta_a: \CE\to \CI_a(\subset \CE)$, $a\in \CA$,  
given by $\theta_a(A)=r(A,a)$, 
every $A\in \CI_a$ satisfies $A\subset r(a)=\theta_a(E^0)$. 
Hence by \cite[Lemma 2.9]{COP:2017}, 
$\theta_a$ induces a continuous map 
$\widehat{\theta}_a: \widehat\CI_a \to \widehat\CE$ given by 
\begin{eqnarray}\label{widehat theta}
\widehat{\theta}_a(\xi)=\{A\in \CE: \theta_a(A)\in \xi\}.
\end{eqnarray}
\end{remark}

\vskip 1pc

Now we review from \cite{COP:2017} 
how one can obtain a topological graph $\BE=(\BE^0,\BE^1,d,r)$ 
from a labeled space $(E,\CL,\CE)$ over a finite alphabet $\CA$. 
Let $(\CE, \CA, \theta)$ be the Boolean dynamical system of 
the labeled space $(E,\CL, \CE)$, namely 
$\theta=(\theta_a)_{a\in \CA}$ and $\theta_a: \CE\to \CE$ 
given by $\theta_a(A):=r(A,a)$ is a Boolean algebra homomorphism 
such that 
$\theta_a(\emptyset)=\emptyset$. 
Set $\BE^0:=\widehat{\CE}$ and 
$\BE^1:=\sqcup_{a\in \CA}\,  \widehat{\CI}_a $
be the disjoint union of Stone's spectrums 
$\{\, \widehat{\CI}_a \,\}_{a\in \CA}$ of the ideals 
$\CI_a$'s.
For convenience, let us write as in \cite{COP:2017} 
$$\BE^0=\{v_\xi:\xi\in \widehat\CE\, \}\ \ \text{and } \ \ 
\BE^1=\sqcup_{a\in \CA}\, \BE^1_a,$$
where $\BE^1_a =\{e^a_\xi: \xi\in  \widehat{\CI}_a \,\}$. 
Define $d, r: \BE^1\to \BE^0$ by 
$$d(e^a_\xi)=v_{\iota(\xi)} \ \ \text{ and }\ \ 
r(e^a_\xi)=v_{\widehat{\theta}_a  (\xi)}.$$
Here $\widehat{\theta}_a : \widehat{\CI}_a  \to \widehat{\CE}$ 
is the map given in (\ref{widehat theta}).
 Then by \cite[Proposition 5.2]{COP:2017} 
$\BE=(\BE^0,\BE^1,d,r)$ is a topological graph which we will 
call the  {\it topological graph} of $(E,\CL)$. 

\vskip 1pc 

\begin{thm} {\rm (\cite[Theorem 5.8 and Example 11.1]{COP:2017})} 
If $(E,\CL)$ is a labeled graph over finite alphabet $\CA$ and  
$\CO(\BE)$ is the $C^*$-algebra of the topological graph 
$\BE$ of $(E,\CL)$, then 
$$C^*(E,\CL) \cong \CO(\BE).$$  
\end{thm}

\vskip 1pc

\subsection{Quasidiagonal $C^*$-algebras} 
 
A separable $C^*$-algebra $A$ is {\it quasidiagonal} 
if it has a faithful representation $\pi:A\to  B(H)$ such that
$\pi(A)$ is a quasidiagonal set of  operators in the sense that
there is an increasing sequence $p_1\leq p_2 \leq \dots$ of 
finite rank projections in $B(H)$ 
such that $p_n\to 1_H$ (SOT) and  
$$\|[\pi(a),p_n]\|\to 0\ \text{ for all } a\in A $$ 
(\cite{Br:2004}). 
It is well known that a quasidiagonal $C^*$-algebra does not have any  
infinite projections, in other words it is a finite $C^*$-algebra. 
Since the property of being quasidiagonal of a $C^*$-algebra $A$ 
is preserved to the matrix algebras $M_n(A)$ over $A$ for all $n\geq 1$, 
it follows that every quasidiagonal $C^*$-algebra is  stably finite,   
whereas the converse is not true. 
AF-embeddable $C^*$-algebras are examples of quasidiagonal $C^*$-algebras.
 
\vskip 1pc 
  
\section{AF-embeddable labeled graph $C^*$-algebras} 

\noindent
In order to see whether the equivalence results of 
Theorem \ref{Pimsner thm} or Theorem \ref{Brown thm} 
can be obtained for a labeled graph $C^*$-algebra, 
we first need to review 
the skew product of labeled graphs from \cite{BPW:2012}. 

Let $(E,\CL)$ be a labeled graph and $c,d:E^1\to G$ be  functions into 
a discrete group $G$. Then 
the {\it  skew product labeled graph} 
$(E\times_c G, \CL_d)$ over alphabet $\CA\times G$ consists of 
the skew product graph $(E^0\times G, E^1\times G, r_c, s_c)$ where 
$$r_c(e,g)=(r(e), gc(e)),\ \ s_c(e,g)=(s(e), g),$$ 
and the labeling map $\CL_d: (E\times_c G)^1\to \CA\times G$ 
given by 
$$\CL_d(e,g):=(\CL(e), gd(e)).$$ 
For example, if 
${\bf 1}:E^1\to G$ denotes the function given by 
${\bf 1}(e):=1_G$, the unit of $G$, for all $e\in E^1$,  
then $\CL_{\bf 1}(e,g)=(\CL(e), g)$ for all $(e,g)\in (E\times_c G)^1$.

\vskip 1pc 

\begin{prop}\label{prop-crossed product} 
Let $(E,\CL,\CE)$ be a weakly left-resolving, set-finite labeled space. 
Then the following are equivalent: 
\begin{enumerate} 
\item[{\rm (a)}] $C^*(E,\CL,\CE)$ is AF embeddable,
\item[{\rm (b)}] $C^*(E,\CL,\CE)$ is quasidiagonal, 
\item[{\rm (c)}] $C^*(E,\CL,\CE)$ is stably finite.
%\item[{\rm (d)}] $C^*(E,\CL,\CE)$ is  finite.
\end{enumerate}
\end{prop}
\begin{proof} 
By Theorem \ref{Brown thm}, it suffices to show that 
$C^*(E,\CL,\CE)$ is stably  isomorphic to the crossed product 
of an AF algebra by $\mathbb Z$.

Define $c:E^1\to \mathbb Z$ by $c(e)=1$ for all $e\in E^1$ and 
consider the $C^*$-algebra  $C^*(E\times_c\mathbb Z, \CL_{\bf 1})$. 
By \cite[Theorem 2.11]{BPW:2012}, this $C^*$-algebra is isomorphic to 
the labeled graph $C^*$-algebra of the labeled graph 
$(E\times_c\mathbb Z, \CL_{\bf 1})$ in our sense, and  
by \cite[Theorem 6.10]{BPW:2012},  
it is stably isomorphic to 
the fixed point algebra $C^*(E,\CL,\CE)^\gm$ of the gauge action $\gm$ on
$C^*(E,\CL,\CE)$, namely 
$$C^*(E\times_c\mathbb Z, \CL_{\bf 1})\otimes \CK\cong 
C^*(E,\CL,\CE)^\gm\otimes \CK$$ 
where $\CK$ denotes the algebra of all compact operators on an  
infinite dimensional separable Hilbert space. 
The fixed point algebra is AF, whence  
$C^*(E\times_c\mathbb Z, \CL_{\bf 1})$ must be an AF algebra. 
Finally from \cite[Corollary 6.8]{BPW:2012} we  see that 
$$C^*(E\times_c\mathbb Z, \CL_{\bf 1})\times_{\tau} \mathbb Z
\cong C^*(E, \CL, \CE)\otimes \CK,$$ 
where 
$\tau:\mathbb Z\to Aut(C^*(E\times_c\mathbb Z, \CL_{\bf 1}))$ is the 
action of $\mathbb Z$ induced by the left labeled graph translation 
action (see Definition 4.5 and Theorem 4.6 of \cite{BPW:2012}). 
\end{proof}
 
\vskip 1pc
  
\section{AF-embeddable $C^*$-algebras of 
labeled graphs over finite alphabet} 

\noindent 
In this section, under the assumptions that 
$E$ is a graph with no sinks or sources and 
$(E,\CL)$ is a labeled graph over finite alphabet 
$\CA=\CL(E^1)$, we provide an equivalent condition for 
$C^*(E,\CL)$ to be quasidiagonal in terms of 
the  infinite path spaces of $(E,\CL)$. 
We also 
show that $C^*(E,\CL)$ is the crossed product of 
the commutative $C^*$-algebra 
$C(X)$ of the continuous functions on a compact metric space $X$ 
($=\overline{\CL(E^{\infty}_{-\infty})}$) 
by $\mathbb Z$ 
generated by the shift map $\sm$ which we define later. 

\vskip 1pc 

\noindent 
{\bf Standing Assumptions:} In this section, $E$ is a 
graph with no sinks or sources and 
$(E,\CL)$ is a labeled graph over finite 
alphabet $\CA:=\CL(E^1)$.

\vskip 1pc

\begin{prop} 
The topological graph $\BE=(\BE^0,\BE^1,d,r)$ of 
$(E,\CL)$ is compact and has no sinks. 
Moreover $d$ is surjective.
\end{prop}
\begin{proof} 
Since $E^0=\cup_{a\in \CA}\,  r(a)\in \CE$,  
we will see that $\BE^0(=\widehat{\CE})$ is compact once we know  
$\widehat{\CE}=Z(E^0)$.
But this is rather obvious 
because every  ultrafilter $\xi\in \widehat{\CE}$ 
must contain the largest set $E^0$ in $\CE$. 

Now we show that $d$ is surjective. 
Let $v_\xi\in \BE^0$, and set
$$\xi_a:=\{B\cap r(a): B\in \xi\}, \ \ a\in \CA.$$ 
We show that there is an $a\in \CA$ such that 
$\xi_a\in \widehat{\CI}_a$ and $d(e^a_{\xi_a})=v_\xi$. 
First note that 
$\emptyset\notin \xi_a$ for some $a\in \CA$. In fact, 
if for every $a\in \CA$ there is a $B_a\in \xi$ such that 
$B_a\cap  r(a)=\emptyset$, 
then with $B:=\cap_a B_a\in \xi$ (which is nonempty since $\xi$ is a filter), 
we have 
$$E^0\cap B=\big( \cup_{b\in \CA} r(b)\big)\cap \big(\cap_{a\in \CA}B_a\big)
=\cup_{b\in \CA} (r(b)\cap (\cap_{a\in \CA} B_a)=\emptyset,$$
a contradiction. 
Choose $a\in \CA$ with $\emptyset\notin \xi_a$. 
Then $\xi_a\in \widehat{\CI}_a$  and thus we can consider an element  
$e^a_{\xi_a}\in \BE^1_a\subset \BE^1$. 
From the definition of $d$,  we have 
$d(e^a_{\xi_a}):= v_{\iota(\xi_a)}=v_{\xi}$ 
if and only if $\iota(\xi_a)=\xi$.
Now $\iota(\xi_a)=\xi$ follows from the fact that the ultrafilter 
$$\iota(\xi_a):=\{A\in \CE: A\supset B \ \text{ for } \exists  B\in \xi_a\}
=\{A\in \CE: A\supset B\cap r(a) \ \text{ for }\exists B\in \xi\}
$$ 
obviously contains the ultrafilter $\xi$. 
\end{proof}

\vskip 1pc

\begin{lem}\label{lem unique vertex} 
Let $\xi\in \widehat\CE$ be a ultrafilter. 
Then for each $l\geq 1$, there exists a unique $[v]_l \in \Om^l(E)$
such that  $[v]_l\in \xi$. 
This unique $[v]_l$  will be denoted $[\xi]_l$.  
\end{lem} 
\begin{proof}
Fix $l\geq 1$. 
Choose any $A:=\cup_{i=1}^m [v_i]_k\in \xi$. 
We may assume that $k\geq l$. 
Since $\xi$ is a ultrafilter, at least one of the generalized vertices 
$[v_i]_k$'s must belong to $\xi$. 
Assume that $[v_1]_k\in \xi$. 
Then $[v_1]_k\subset [v_1]_l$ and 
this implies $[v_1]_l\in \xi$. 
From $\emptyset\notin \xi$, we see that $\xi$ does not 
contain any vertex  $[v]_l$ other than $[v_1]_l$ because 
$[v]_l\cap [v_1]_l$ is either $[v_1]_l$ or empty.
\end{proof}

\vskip 1pc

If $d: \BE^1\to \BE^0$ is injective, 
then $d(e^a_\xi)=d(e^b_\eta)$ implies $a=b$ (hence $\xi,\eta\in \widehat{\CI}_a$) 
and $\iota(\xi)=\iota(\eta)$. 
Since $\xi\mapsto \iota(\xi): \widehat{\CI}_a \to \widehat\CE$ is injective, 
we also have  $\xi=\eta$.   

\vskip 1pc 

\begin{prop} \label{disjoint range} 
Let $(\BE^0,\BE^1, d,r)$ be the topological graph of 
$(E,\CL,\CE)$. Then 
$d:\BE^1\to \BE^0$ is injective 
if and only if 
\begin{eqnarray}\label{eq disjoint range} 
r(a)\cap  r(b)=\emptyset \ \ \text{ for } \ a\neq b.
\end{eqnarray} 
In particular, (\ref{eq disjoint range}) holds  if $C^*(E,\CL,\CE)$ is finite.
\end{prop} 
\begin{proof}
Let $d:\BE^1\to \BE^0$ be injective. 
Suppose  $r(a)\cap r(b)\neq \emptyset$ for some $a,b\in \CA$ with $a\neq b$.
Then for $A_{ab}:=r(a)\cap r(b)\in \CE$, 
the ideal $\CI_{A_{ab}}=\{B\in \CE: B\subset A_{ab}\}$ 
is nonempty.  
Choose $\xi\in \widehat{\CI}_{A_{ab}}$ and consider   
$\iota (\xi)$ in $\widehat{\CE}$. 
Then by (\ref{ideals}),
$$\iota (\xi)=\iota (\iota_{\widehat{\CI}_a}(\xi))=
\iota (\iota_{\widehat{\CI}_b}(\xi)).$$  
Therefore  $d$ maps the two different edges 
$e^a_{\iota_{\widehat{\CI}_a}(\xi)}$ and 
$e^b_{\iota_{\widehat{\CI}_b}(\xi)}$ in $\BE^1$ (since $a\neq b$) 
to the same vertex $v_{\iota (\xi)}$ in $\BE^0$, 
a contradiction to $d$ being injective. 

For the converse, assume that 
$r(a)\cap r(b)=\emptyset$ whenever $a\neq b$. 
To show the injectivity of $d$, suppose to the contrary that 
$$d(e^a_{\xi})=d(e^b_{\eta})\ \text{ for }\ 
e^a_{\xi}\neq e^b_{\eta}.$$ 
Then we have 
$\iota(\xi)=\iota(\eta)$, namely
$$\{A\in \widehat\CE: A\supset B\text{ for } \exists B\in \xi\}
=\{A\in \widehat\CE: A\supset B\text{ for } \exists B\in \eta\}.$$ 
If $a=b$, then the injectivity of $\iota$ gives 
$\xi=\eta$, which is not possible since $e^a_{\xi}\neq e^b_{\eta}$. 
So suppose  $a\neq b$. 
If $B\in \xi$, then $B\subset r(a)$ and 
so $B\cap B'=\emptyset$ for any $B'\in \eta$ since 
$B'\subset r(b)$, which means $B\notin \iota(\eta)$. 
But this is also impossible because $\xi\subset \iota(\xi)=\iota(\eta)$.  
 
The last assertion follows from the fact that 
if $C^*(E,\CL)\cong C^*(\BE)$ is finite, $d$ is injective  
(see \cite[Theorem 6.7]{Sch:2015}). 
\end{proof}

\vskip 1pc

\noindent
From the proof of the above proposition, we see that 
if $C^*(E,\CL)$ is finite,  
the Stone's spectrum of $\CE$ is the disjoint union of 
the compact open subsets $Z(r(a))$'s, that is  
$\widehat\CE=\sqcup_{a\in \CA}\, Z(r(a))$. Moreover 
we have the following corollary. 

\vskip 1pc 

\begin{cor} \label{cor disjoint path}
For the topological graph $(\BE^0,\BE^1, d,r)$  of $(E,\CL)$,  
the following are equivalent: 
\begin{enumerate}
\item[{\rm(a)}]  
$d:\BE^1\to \BE^0$ is injective.
\item[{\rm (b)}] 
For $\af,\bt\in \CL^*(E)$,  
$r(\af)\cap r(\bt)\neq \emptyset$ if and only if either 
$\af\in (\bt]  \text{ or } \bt\in (\af]$.
\item[{\rm (c)}] For each $[v]_l\in \Om^l(E)$, there exists a unique 
$a_l\cdots a_1\in \CL(E^l)$ such that $[v]_l=r(a_l\cdots a_1)$.
\end{enumerate}
\end{cor} 
\begin{proof} 
$(a)\Rightarrow (b)$ 
Suppose $d$ is injective, or equivalently 
$r(a)\cap r(b)=\emptyset$ whenever $a\neq b$ in $\CA$, 
and let $\af,\bt\in \CL^*(E)$ have   
$r(\af)\cap r(\bt)\neq \emptyset$ . 
Writing $\af=\af'a$ and $\bt=\bt'b$ for some 
$\af', \bt'\in \CL^\#(E)$ gives $a=b$ immediately since 
$r(\af)\cap r(\bt)\subset r(a)\cap r(b)$. 
Also we see, from the assumption that $(E,\CL,\CE)$ is
weakly left-resolving,  that
$$r(\af)\cap r(\bt)=r(r(\af'),a)\cap r(r(\bt'),a)
=r(r(\af')\cap r(\bt'), a)\neq \emptyset$$ 
only when $\af'$ and $\bt'$ must end with the same letter. 
The same argument can be continued to conclude the assertion. 

\noindent
$(b)\Rightarrow (a)$ follows from Proposition \ref{disjoint range}.

\noindent
$(b)\Leftrightarrow (c)$ is obvious since $[v]_l$ is the intersection of 
ranges of finitely many paths in $\CL(E^l)$.
\end{proof}

\vskip 1pc
  
By $[\xi]_l$ we denote the unique generalized vertex 
$[v]_l$ in $\xi\in \widehat\CE$ obtained in Lemma \ref{lem unique vertex} . 
Then  it is easily checked that 
$$Z([v]_l)=\{\xi\in \widehat\CE: [\xi]_l=[v]_l\},$$ 
and the map 
\begin{equation} 
\xi\mapsto ([\xi]_1, [\xi]_2, \dots): 
\widehat{\CE}\to \underset{l\geq 1}{\prod} \Om^l(E)
\end{equation} 
is one-to-one (not necessarily onto). 
Moreover one can show that 
for an ultrafilter $\xi\in \widehat{\CI}_a$ and $l\geq 1$, 
there is a unique $[v]_l$ contained in $\xi(\subset \iota(\xi))$. 
Thus we may write $[\xi]_l$ for $[\iota(\xi)]_l$ 
with no risk of confusion. 

Now define a metric $\rho$ on $\widehat{\CE}$ as follows:  
$$\rho(\xi,\eta):=2^{-n}, \ \text{ where } 
n:=\min\{l\geq 1: [\xi]_l\neq [\eta]_l\}.$$ 
Then a simple calculation gives
\begin{eqnarray}\label{topology} 
{\rm ball}(\xi:\frac{1}{2^l}):=
\{\eta\in \widehat\CE: \rho(\xi,\eta)< \frac{1}{2^l}\}=
\{\eta\in \widehat\CE: \rho(\xi,\eta)\leq \frac{1}{2^{l+1}}\}=Z([v]_l), 
\end{eqnarray} 
which proves the following lemma.
 
\vskip 1pc 

\begin{lem} 
If $\BE=(\BE^0,\BE^1)$ is the topological graph of $(E,\CL)$, 
the metric $\rho$ is compatible with the topology on 
$\BE^0(=\widehat\CE\,)$. 
\end{lem}
 
\vskip 1pc

Recall from \cite[Definition 6.6]{Sch:2015} that  
an {\it $\varepsilon$-pseudopath}, $\varepsilon>0$, 
in a topological graph $\BE=(\BE^0,\BE^1, r,d)$ is a finite sequence 
$\af=(e_n, \dots, e_1)$ of edeges in $\BE$ such that 
for each $i=1,\dots, n-1$, 
$$\rho(r(e_i), d(e_{i+1}))<\varepsilon,$$ 
where $\rho$ is a metric compatible with the topology of $\BE^0$. 
An $\varepsilon$-pseudopath is called an 
{\it $\varepsilon$-pseudoloop} based at 
$d(\af):=d(e_1)$ if 
$\rho(r(\af), d(\af))<\varepsilon$ where $r(\af):=r(e_n)$. 
It is known in \cite{Sch:2015} that 
if $\BE$ is a compact topological graph with no sinks such that 
the $C^*$-algebra $C^*(\BE)$ is AF-embeddable  (or equivalently 
finite), then for each $v\in \BE^0$ and $\varepsilon>0$, 
there exists an $\varepsilon$-pseudoloop based at $v$. 

We are now interested in finding conditions of a labeled space 
$(E,\CL)$ that gives rise to a quasidiagonal $C^*$-algebra 
$C^*(E,\CL)$ which can be viewed as the $C^*$-algebra of 
the topological graph $\BE$ of $(E,\CL,\CE)$. 
Thus we need to understand when an `$\varepsilon$-pseudoloop' exists
at each $\xi\in \widehat{\CE}=\BE^0$ 
in terms of labeled space $(E,\CL)$.  
For this, we start with measuring the distance between  range and 
domain vertices of two edges in $\BE^1$. 
\vskip 1pc 

\begin{lem} \label{lem distance of vertices}
Let $(\BE^0,\BE^1, d,r)$ be the topological graph of $(E,\CL)$ 
and let $e^a_\xi$ and $e^b_\eta$ be edges in $\BE^1$. Then
$$\rho\big( d(e^a_\xi), r(e^b_\eta)\big)\leq \frac{1}{2^{(m+1)}} 
\ \text{ if and only if }\ r([\xi]_l,b)=[\eta]_{l+1} 
\, \text{ for } 1\leq l \leq m.$$ 
Moreover, 
$d(e^a_\xi)=r(e^b_\eta)$ if and only if 
$r([\xi]_l,b)=[\eta]_{l+1}$ for $l\geq 1.$
\end{lem}
\begin{proof}
By definition of $d$ and $r$,  
$\rho ( d(e^a_\xi), r(e^b_\eta))\leq \frac{1}{2^{(m+1)}}$ 
holds if and only if  
\begin{eqnarray}\label{distance ineq}
\rho(\iota(\xi), \widehat\theta_b(\eta))\leq \frac{1}{2^{(m+1)}},
\end{eqnarray} 
or $[\iota(\xi)]_l=[\widehat\theta_b(\eta)]_l$ 
for $1\leq l\leq  m$.
Since  
$$[\widehat\theta_b(\eta)]_l\in 
\widehat\theta_b(\eta) =\{A\in \widehat\CE: r(A,b)\in \eta\}$$ 
and $[\iota(\xi)]_l=[\xi]_l$, we have 
$r([\xi]_l,b)\in \eta$. 
But we can write $[\xi]_l=\cap_j r(\dt_j)$ 
for some paths $\dt_j\in \CL(E^l)$ of length $l$, 
and then 
$r([\xi]_l,b)=\cap_j r(\dt_j b)$ which can be written as 
a union of finitely many vertices in $\Om^{l+1}(E)$. 
Thus there is a $v\in E^0$ such that  
$[v]_{l+1}\subset r([\xi]_l,b)\in \eta$.
Then Lemma \ref{lem unique vertex} says that  
$[v]_{l+1}=[\eta]_{l+1}$. 
Since $\eta\in \widehat{\CI}_b$, we can write 
$[v]_{l+1}=[\eta]_{l+1}=\cap_j r(\af_jb)$ for 
some paths $\af_j$'s of length $l$. 
Then  
$$ \cap_j r(\af_jb)=
\cap_j r(r(\af_j),b)=r(\cap_j r(\af_j), b)
=r(\cup_i [w_i]_l,b)$$ 
 for some vertices  $[w_i]_l\in\Om^l(E)$ 
 because $\cap_j r(\af_j)$ is a finite union of 
vertices in $\Om^{l}$, which shows that 
 $[v]_{l+1}$ must be of the form  
$\cap_i\, r([w_i]_l, b)$, and 
$$r([\xi]_l,b)\supset [v]_{l+1}=r(\cup_i [w_i]_l,b)
= \cup_i r( [w_i]_l,b).$$ 
Thus we can conclude  that 
$[v]_{l+1}=\cup_i r( [w_i]_l,b)=r([w_1]_l,b)$ and 
$r([\xi]_l,b)\supset [v]_{l+1}= r( [w_1]_l,b)$. 
Therefore $[\xi]_l=[w_1]_l$, and hence it follows that 
$r([\xi]_l,b)=[v]_{l+1} =[\eta]_{l+1}$.

For the converse, 
suppose $r([\xi]_l,b)=[\eta]_{l+1}$ for $1\leq l \leq m$.  
Then $[\xi]_l\in \widehat\theta_b(\eta)$,
hence again by Lemma \ref{lem unique vertex}, 
$[\xi]_l=[\widehat\theta_b(\eta)]_l$ is obtained 
for $1\leq l\leq  m$. 
Thus we have (\ref{distance ineq}).
\end{proof}

\vskip 1pc 

\begin{lem}\label{pseudoloop}
Let $(\BE^0,\BE^1, d,r)$ be the topological graph of $(E,\CL)$ 
with an injective $d:\BE^1\to \BE^0$. 
Then for a sequence $\af=(e^{a_n}_{\xi_n}, \dots, e^{a_1}_{\xi_1})$in $\BE^1$, 
$\af$ is a 
$\frac{1}{2^{m+2}}$-pseudoloop at $\xi_1 \in \BE^0(=\widehat\CE)$ 
if and only if  
for each $1\leq i\leq  n$,
\begin{eqnarray} \label{ps-loop} 
[\xi_i]_m=r(a_{i+m-1}\cdots a_{i+1}a_i),
\end{eqnarray} 
where 
$a_j:=a_{j'}$ for $j=j'$ {\rm (mod $n$)}. 

\end{lem}
\begin{proof} 
By Lemma \ref{lem distance of vertices}, we easily see that 
(\ref{ps-loop}) is sufficient for  $\af$ to be a 
$\frac{1}{2^{m+2}}$-pseudoloop at $\xi_1$. 

So we assume for the converse 
that $\af$ is a $\frac{1}{2^{m+2}}$-pseudoloop at $\xi_1$. 
Then  
\begin{align*} 
\rho\big(d(e^{a_{i+1}}_{\xi_{i+1}}), r(e^{a_{i}}_{\xi_{i}})\big)
& \leq \frac{1}{2^{m+1}}
\ \text{ for } 1\leq i \leq n-1,\\ 
\rho\big(d(e^{a_{1}}_{\xi_{1}}), r(e^{a_{n}}_{\xi_{n}})\big)
& \leq \frac{1}{2^{m+1}},
\end{align*}
which is, by Lemma \ref{lem distance of vertices}, 
equivalent to the following: 
$$r([\xi_{i+1}]_l,a_i) =[\xi_i]_{l+1} 
\ \ {\rm for } \ \ 1\leq i \leq n\ ({\rm mod}\ n),\ 1\leq l\leq m.$$
From Corollary \ref{cor disjoint path} and 
the fact that $\xi_i\in \widehat{\CI}_{a_i}$, we have 
$$[\xi_i]_1=r(a_i) \ \text{ for }1\leq i\leq n \,({\rm mod}\ n).$$
Hence 
$[\xi_i]_2=r([\xi_{i+1}]_1,a_i)=r(r(a_{i+1}),a_i)=r(a_{i+1} a_i)$ 
for $1\leq i\leq n\, ({\rm mod}\ n)$.
(\ref{ps-loop}) is then obtained by induction on 
$l$, $1\leq l\leq m$. 
\end{proof}

\vskip 1pc 

The following lemma is an easy observation. 

\vskip 1pc

\begin{lem}\label{ultrafilter of vertices} 
Let $\{[v_k]_k\}_{k=1}^\infty$ be a decreasing sequence of 
generalized vertices in $\CE$. Then 
$$\xi:=\{A\in \CE: A\supset [v_k]_k \ \text{ for some } k\geq 1\}$$ 
is a ultrafilter in $\widehat\CE$, and 
$[\xi]_k=[v_k]_k$ for all $k\geq 1$. 
\end{lem}

\vskip 1pc 
\noindent 

\begin{remarks} \label{remark finite sequence}
(a) By the previous lemma, we can identify $\widehat\CE$ with 
the set of all decreasing sequences $\{[v_k]_k\}_{k=1}^\infty$ 
of generalized vertices. 

(b) If 
$\{[w_k]_k\}_{k=1}^m$ is a finite sequence of generalized vertices, 
there exist infinite decreasing sequences 
$\{[v_l]_l\}_{l\geq 1}$ of generalized vertices such that $[v_k]_k=[w_k]_k$   
for $1\leq k\leq m$. Thus from the sequence $\{[w_k]_k\}_{k=1}^m$ 
we obtain ultrafilters $\xi\in \widehat\CE$ such that 
$[\xi]_k=[w_k]_k$ for $1\leq k\leq m$ as mentioned in (a).
\end{remarks}

\vskip 1pc 

\begin{lem} \label{lem main}
Let $(\BE^0,\BE^1, d,r)$ be the topological graph of $(E,\CL)$
and let $d:\BE^1\to \BE^0$ be injective.  
If $\xi\in \widehat\CE$, $m\geq 1$, and  $a_1, \dots, a_n\in\CA$ satisfy 
\begin{enumerate} 
\item[{\rm (i)}] $[\xi]_m=r(a_m\cdots a_1)$, 
\item[{\rm (ii)}] $a_{i+m-1}\cdots a_i\in \CL(E^m)$ for each $1\leq i\leq n$, 
\end{enumerate} 
with the convention that $a_l=a_{l'}$ whenever $l=l'\, ({\rm mod}\, n)$, 
then there exists a $\frac{1}{2^{m+2}}$-pseudoloop based at $\xi$. 
\end{lem} 
\begin{proof}
First note that $[\xi]_k=r(a_k\cdots a_1)$ holds for 
all $1\leq k\leq m$.
Set 
$$[v_i]_k:= r(a_{i+k-1}\cdots a_i)$$ 
for $1\leq i\leq n$ and $1\leq k\leq m$, 
which is well-defined by Corollary \ref{cor disjoint path}.
Then for the decreasing finite sequence 
$[v_i]_k\supset \dots \supset [v_i]_m$, one can 
choose $\xi_i\in \widehat\CE$ such that 
$[\xi_i]_k=[v_i]_k$, $1\leq k\leq m$ 
(see Remarks \ref{remark finite sequence}.(b)).
Then 
$\af:=e^{a_n}_{\xi_n}\cdots e^{a_1}_{\xi_1}$ is a 
 $\frac{1}{2^{m+2}}$-pseudoloop based at $\xi(=\xi_1)$ 
by Lemma \ref{pseudoloop}. 
\end{proof}
 
\vskip 1pc

\begin{dfn} \label{def pseudo-periodic}
We say that $(E,\CL)$ is {\it pseudo-periodic} if  
for each  $m>1$ and $a_{m-1}\cdots a_1\in \CL(E^m)$,
there exist an $n\geq 1$ and finite paths 
$a_{i+m-1}\cdots a_{i}\in \CL(E^m)$ $(1\leq i\leq n+1)$  
such that 
$$a_{n+m-1}\cdots a_{n+1}=a_{m-1}\cdots a_1.$$
Note that we do not require   
$a_{n+m-1}\cdots a_{n+1} a_n\cdots a_1\in \CL^*(E)$.
\end{dfn} 

\vskip 1pc

\begin{prop} \label{prop path chain} 
Let $(\BE^0,\BE^1, d,r)$ be the topological graph of $(E,\CL)$ 
and $d:\BE^1\to \BE^0$ be injective.   
If $(E,\CL)$ is  pseudo-periodic, 
then for each $\xi\in  \BE^0$ and $\varepsilon>0$, 
there exists an $\varepsilon$-pseudoloop based at $\xi$. 
\end{prop}
\begin{proof} 
Let $\xi\in \BE^0$ and $\varepsilon>0$. 
Fix $m> 1$ such that $1/2^{m}<\varepsilon$ and  
let $a_m\cdots a_1$ be the path with $[\xi]_m=r(a_m\cdots a_1)$. 
Since $(E,\CL,\CE)$ is pseudo-periodic, 
there exist paths $a_{i+m-1}\cdots a_{i}\in \CA$, 
$1\leq i\leq n+1$ 
such that 
$$a_{n+m-1}\cdots a_{n+1}=a_{m-1}\cdots a_1.$$
Set $\xi_1=\xi$ and for each $2\leq i\leq n+1$ 
choose $\xi_i\in \BE^0$ 
such that $$[\xi_i]_{m}=r(a_{i+m-1}\dots a_i).$$ 
(For the existence of such a ultrafilter $\xi_i$, 
see Remarks \ref{remark finite sequence}.(b).)
Then for the edges $e^{a_i}_{\xi_i}\in \BE^1$, $1\leq i\leq n-1$, 
one can check that 
$$[r(e^{a_i}_{\xi_i})]_{m-1}=r(a_{i+m-1}\cdots a_{i+1})
=[d(e^{a_{i+1}}_{\xi_{i+1}})]_{m-1},$$ 
and also for $i=n$, 
$$[r(e^{a_n}_{\xi_n})]_{m-1}= r(a_{n+m-1}\cdots a_{n+1})=
r(a_{m-1}\cdots a_1)=[d(e^{a_1}_{\xi_1})]_{m-1}.$$
This shows that  
$\rho(r(e^{a_i}_{\xi_i}), d(e^{a_{i+1}}_{\xi_{i+1}}))
\leq 1/2^m \,(<\varepsilon)$ for $1\leq i\leq n-1$, 
and also $\rho(r(e^{a_n}_{\xi_n}), d(e^{a_{1}}_{\xi_{1}}))
\leq 1/2^m \,(<\varepsilon)$. 
Therefore the sequence $(e^{a_n}_{\xi_n}, \dots, e^{a_1}_{\xi_1})$ 
in  $\BE^1$  is an $\varepsilon$-pseudoloop at $\xi$. 
\end{proof}

\vskip 1pc

\begin{prop} \label{prop path} 
Let $(\BE^0,\BE^1, d,r)$ be the topological graph of $(E,\CL)$ 
and let $d:\BE^1\to \BE^0$ be injective.    
Assume further that for any left-infinite 
sequence $(\dots, a_{3}, a_2,a_1)\in \overline{\CL(E_{-\infty})}$ 
and $m\geq 1$, there exists an $n\geq 1$ such that 
\begin{eqnarray}\label{n}
a_{n+m}\cdots a_{n+1}=a_m\cdots a_1.
\end{eqnarray}
Then $(E,\CL)$ is pseudo-periodic.
\end{prop}
\begin{proof} 
Let $m>1$ and $a_m\cdots a_1\in \CL(E^m)$. 
Pick $\xi\in \BE^0$ with $[\xi]_m=r(a_m\cdots a_1)$. 
Then by Lemma \ref{lem unique vertex} and 
Corollary \ref{cor disjoint path}, for each $l\geq 1$ 
there exists a unique path  $a_l a_{l-1}\cdots a_1$ such that 
 $[\xi]_l=r(a_l a_{l-1}\cdots a_1)$. 
Setting $l\to \infty$, we ontain a  
 left-infinite path $\cdots a_la_{l-1}\cdots a_1$ in  $\CA^{-\mathbb N}$
such that each of its finite words must appear as a labeled path in $\CL^*(E)$. 
Particularly, $a_{i+m-1}\cdots a_{i}\in \CL(E^m)$ for all $i\geq 1$. 
By the assumption, there is an $n\geq 1$ such that 
$a_{n+m}\cdots a_{n+1}=a_m\cdots a_1$, which proves that 
$(E,\CL)$ is pseudo-periodic.
\end{proof}
 
\vskip 1pc
\noindent 
The converse of Proposition \ref{prop path} is not true, in general, 
as we see in the following example. 

\vskip 1pc 

\begin{ex}\label{ex one} 
 Consider the following labeled graph $(E,\CL)$ with 
$\CA=\{0,1\}$.

\vskip 1.5pc 
\hskip -.5cm
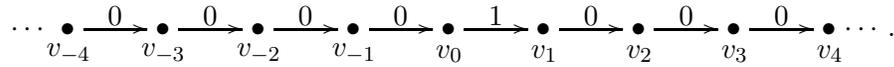
\xy  /r0.3pc/:(-44.2,0)*+{\cdots};(44.3,0)*+{\cdots .};
(-40,0)*+{\bullet}="V-4";
(-30,0)*+{\bullet}="V-3";
(-20,0)*+{\bullet}="V-2";
(-10,0)*+{\bullet}="V-1"; (0,0)*+{\bullet}="V0";
(10,0)*+{\bullet}="V1"; (20,0)*+{\bullet}="V2";
(30,0)*+{\bullet}="V3";
(40,0)*+{\bullet}="V4";
 "V-4";"V-3"**\crv{(-40,0)&(-30,0)};
 ?>*\dir{>}\POS?(.5)*+!D{};
 "V-3";"V-2"**\crv{(-30,0)&(-20,0)};
 ?>*\dir{>}\POS?(.5)*+!D{};
 "V-2";"V-1"**\crv{(-20,0)&(-10,0)};
 ?>*\dir{>}\POS?(.5)*+!D{};
 "V-1";"V0"**\crv{(-10,0)&(0,0)};
 ?>*\dir{>}\POS?(.5)*+!D{};
 "V0";"V1"**\crv{(0,0)&(10,0)};
 ?>*\dir{>}\POS?(.5)*+!D{};
 "V1";"V2"**\crv{(10,0)&(20,0)};
 ?>*\dir{>}\POS?(.5)*+!D{};
 "V2";"V3"**\crv{(20,0)&(30,0)};
 ?>*\dir{>}\POS?(.5)*+!D{};
 "V3";"V4"**\crv{(30,0)&(40,0)};
 ?>*\dir{>}\POS?(.5)*+!D{};
 (-35,1.5)*+{{0}};(-25,1.5)*+{{0}};
 (-15,1.5)*+{{0}};(-5,1.5)*+{{0}};(5,1.5)*+{1};
 (15,1.5)*+{0};(25,1.5)*+{{0}};(35,1.5)*+{{0}};
 (0.1,-2.5)*+{v_0};(10.1,-2.5)*+{v_1};
 (-9.9,-2.5)*+{v_{-1}};
 (-19.9,-2.5)*+{v_{-2}};
 (-29.9,-2.5)*+{v_{-3}};
 (-39.9,-2.5)*+{v_{-4}}; 
 (20.1,-2.5)*+{v_{2}};
 (30.1,-2.5)*+{v_{3}};
 (40.1,-2.5)*+{v_{4}}; 
\endxy 
\vskip 1.5pc 

\noindent 
We first show that the labeled space is pseudo-periodic. 
Let $m>1$ and $a_m\cdots a_1\in \CL(E^m)$. 
We have to find an $n\geq 1$ and paths $a_{i+m-1}\cdots a_i\in \CL(E^m)$ 
for $1\leq i\leq n$ such that 
$a_{n+m-1}\cdots a_{n+1}=a_{m-1}\cdots a_1$. 
If $a_m\cdots a_1=0\cdots 0$, then any $n\geq 1$ will do the job 
with $a_{i+m-1}\cdots a_i=0\cdots 0$ for $1\leq i\leq n$. 
If 
$$a_m\cdots a_1=0\cdots 1,$$
then taking $n=m$ and  
$a_{m+m}\cdots a_{m+1}:=0\cdots 1$ we have
that $a_{i+m-1}\dots a_i \in \CL^*(E)$ for 
$1\leq i\leq n$  and 
$a_{n+m}\cdots a_{n+1}=0\cdots 1=a_m\cdots a_1$. 
Finally if $a_m\cdots a_1=0\cdots 1\cdots 0$ with $a_k=1$ 
for some $k$ and 
$a_j=0$ for $j\neq k$, then  
take $n=m$ and $a_{m+j}:=a_j$ for $1\leq j\leq m$. 
It is then easy to see that 
$a_{i+m-1}\cdots a_i\in \CL(E^m)$ for $1\leq i\leq n$,
and 
$ a_{n+m-1}\cdots a_{n+1}=a_{m-1}\cdots a_1$.
Thus $(E,\CL)$ is pseudo-periodic. 

Note that 
any left-infinite sequence $(\cdots, a_3,a_2,a_1)$ whose finite blocks 
belong to $\CL^*(E)$ can have at most one $a_i$ which is equal to $1$.  
Thus for the sequence 
$$(\cdots, a_3,a_2,a_1):=(\dots, 0,0,0,1),$$
while each of its finite blocks appears in $\CL^*(E)$, 
there is no $n\geq 1$ such that 
$a_{n+m-1}\cdots a_{n+1}=a_m\cdots a_1$ 
holds, namely this labeled space does not satisfy the assumption 
of Proposition \ref{prop path}.
\end{ex}

\vskip 1pc 

\begin{thm}\label{thm main} 
Let $E$ have no sinks or sources 
and $(E,\CL)$ be a labeled graph over finite alphabet $\CA$. 
Then the following are equivalent: 
\begin{enumerate} 
\item[{\rm (a)}] $C^*(E,\CL)$  is AF-embeddable;
\item[{\rm (b)}] $C^*(E,\CL)$  is quasidiagonal;
\item[{\rm (c)}] $C^*(E,\CL)$  is stably finite;
\item[{\rm (d)}] $C^*(E,\CL)$  is finite.
\item[{\rm (e)}] $(E,\CL)$ is pseudo-periodic, and for $\af,\bt\in \CL(E^k)$, $k\geq 1$,   
$$r(\af)\cap r(\bt)\neq \emptyset \iff \af=\bt.$$ 
\end{enumerate}
\end{thm}
\begin{proof}  
The first four conditions are known equivalent in \cite[Theorem 6.7]{Sch:2015}. 
Also the path range condition in (e) is equivalent to the injectivity of 
$d: \BE^1\to \BE^0$ by  Corollary \ref{cor disjoint path}, 
and the pseudo-periodic condition is equivalent to the existence of 
an $\varepsilon$-pseudoloop at any $\xi\in \widehat\CE$ and any $\varepsilon>0$ 
by Lemma \ref{pseudoloop} and Proposition \ref{prop path chain}.
\end{proof}

\vskip 1pc

If $d:\BE^1\to \BE^0$ is injective 
and $\xi\in \BE^0$ is a ultrafilter in $\CE$, 
there exists a unique 
decreasing sequence $[\xi]_1\supset [\xi]_2\supset \cdots$ of 
generalized vertices $[\xi]_n$'s each of which is a range of 
a unique path of length $n$. 
We will denote this path by $a_\xi^n \cdots a_\xi^1$, namely 
$$[\xi]_n=r(a_\xi^n \cdots a_\xi^1)$$ for all $n\geq 1$. 
Thus the map $\xi\mapsto {\bf a}_\xi$, where 
$${\bf a}_\xi:=\cdots a_\xi^n \cdots a_\xi^1\in \CA^{-\mathbb N}$$ 
is a left infinite sequences such that 
$a_\xi^n \cdots a_\xi^1\in \CL^*(E)$ for each $n\geq 1$, 
is a one-to-one correspondence between 
the ultrafilters in $\widehat{\CE}$ and $\overline{\CL(E_{-\infty})}$. 

On the set $\CA^{\mathbb Z}$ of all bi-infinite words  
which are the functions ${\bf a}: \mathbb Z\to \CA$, 
$a_i:={\bf a}(i)$,
we will consider the shift transform 
$\tau:\CA^{\mathbb Z}\to \CA^{\mathbb Z}$  given by 
$$ \tau({\bf a})(i)={\bf a}(i+1), \ i\in \mathbb Z,$$ 
namely,  
$\tau(\cdots a_{-1}.a_{0}a_1\cdots)=\cdots a_{-1}a_{0}.a_1\cdots.$ 
Then $\tau$ is a homeomorphism on $\CA^{\mathbb Z}$. 
Moreover, the restriction of $\tau$  to the path space 
$\overline{\CL(E_{-\infty}^\infty)}$ 
is a homeomorphism onto itself.

Note also from \cite[Theorem 6.4]{Sch:2015} that 
we have a homeomorphism
$\sm: \BE^\infty\to\BE^\infty$ given by 
$$\sm(\cdots e^{a_n}_{\xi_n}\cdots e^{a_1}_{\xi_1})
=\cdots e^{a_n}_{\xi_n}\cdots e^{a_2}_{\xi_1}.$$
We will also use $\sm$ for the map on $\CA^{-\mathbb N}$ 
given by
$$\sm(\cdots a_3a_2a_1):=\cdots a_3a_2$$ to avoid too much notation.

\vskip 1pc

\begin{lem}\label{infinite path space} 
Let $(\BE^0,\BE^1, d,r)$ be the topological graph of $(E,\CL)$ 
and let $d:\BE^1\to \BE^0$ be injective.   
Then 
there is a homeomorphism 
$$x\mapsto {\bf a}_x: 
\BE^\infty\to \overline{\CL(E^{\,\infty}_{-\infty})}$$ 
between two compact spaces. Moreover 
${\bf a}_{\sm(x)}=\tau({\bf a}_x)$ for $x\in \BE^\infty$.
\end{lem} 
\begin{proof} 
We first claim that 
for $e^a_\xi, e^b_\eta\in\BE^1$,   
\begin{eqnarray}\label{edges} 
e^b_\eta e^a_\xi\in \BE^2 \ \Leftrightarrow\ 
d(e^a_\xi)=r(e^b_\eta)\ \Leftrightarrow\ 
\sm({\bf a}_\eta)={\bf a}_\xi.
\end{eqnarray} 
(Note here that $a_{\eta}^2 a_{\eta}^1 =ab$.) 
The first equivalence is just the definition of paths of length 2 in 
$\BE$ and thus we only need to show the second.
Since $r(e^b_\eta)=v_{\widehat\theta_b (\eta)}$, by 
Lemma \ref{lem unique vertex} and Lemma \ref{ultrafilter of vertices} 
$$[\iota(\xi)]_n=[\widehat\theta_b (\eta)]_n$$
for all $n\geq 1$. 
Then $r(a_\xi^n \cdots a_\xi^1)=[\xi]_n=[\widehat\theta_b (\eta)]_n
\in \{A\in \CE: r(A, b)\in \eta\}$, $n\geq 1$, implies that 
$r(a_\xi^n \cdots a_\xi^1b)=r(r(a_\xi^n \cdots a_\xi^1),b) \in \eta$. 
But this is a set of the form $[v]_{n+1}$ in $\eta$, hence 
must coincide with $r(a_\eta^{n+1}\cdots a_\eta^2 a_\eta^1)$. 
By Proposition \ref{disjoint range} we obtain 
$$a_\xi^n \cdots a_\xi^2 ab=
a_\xi^n \cdots a_\xi^1b=a_\eta^{n+1}\cdots a_\eta^2 a_\eta^1$$ 
for all $n\geq 1$, 
which shows that $d(e^a_\xi)=r(e^b_\eta)$ if and only if 
 $\sm({\bf a}_\eta)={\bf a}_\xi$ (and $ab=a_\eta^2 a_\eta^1\in \CL^*(E)$).
 
Now if $x =\cdots e^{a_n}_{\xi_n}\cdots e^{a_1}_{\xi_1}\in \BE^\infty$, 
by the above claim we see that  
\begin{eqnarray}\label{bi-infinite seq} 
\sm^n({\bf a}_{\xi_{n+1}})= {\bf a}_{\xi_1}
\end{eqnarray}
and 
$a_1\cdots a_n\in \CL^*(E)$ for all $n\geq 1$. 
Let 
$${\bf a}_{\xi_1}
=\cdots a^n_{\xi_1}\cdots a^2_{\xi_1}a^1_{\xi_1}, \ \ 
 a^1_{\xi_1}=a_1.$$  
Then by (\ref{bi-infinite seq}) with $n=2$, 
$${\bf a}_{\xi_1}=
\sm({\bf a}_{\xi_2})=\sm(\cdots a^n_{\xi_2}\cdots a^2_{\xi_2}a^1_{\xi_2} )
= \cdots a^n_{\xi_2}\cdots a^3_{\xi_2}a^2_{\xi_2}.$$ 
Hence $a^{i+1}_{\xi_2}=a^i_{\xi_1}$ for all $i\geq 1$.
Moreover $a^1_{\xi_1}=a_1$ and $a^1_{\xi_2}=a_2$ give 
\begin{eqnarray} \label{label seq}
{\bf a}_{\xi_2}=\cdots a^3_{\xi_1} a^2_{\xi_1}a_1a_2.
\end{eqnarray}
Thus an induction used to obtain (\ref{label seq}) shows that 
for all $n\geq 1$, 
\begin{eqnarray} \label{xi_n}
{\bf a}_{\xi_n}=\cdots a^3_{\xi_1} a^2_{\xi_1}a_1a_2\cdots a_n.
\end{eqnarray}
So far we have shown that each 
$x=\cdots e^{a_n}_{\xi_n}\cdots e^{a_1}_{\xi_1}$ in 
$\BE^\infty$ 
defines a  bi-infinite sequence
 $${\bf a}_{x}:=\cdots a^3_{\xi_1} a^2_{\xi_1}.a_1a_2\cdots a_n\cdots$$ 
in $\overline{\CL(E^{\,\infty}_{-\infty})}$. 

For the converse, let ${\bf a}:= 
{\bf a}_{(-\infty,0]} a_1a_2\cdots\in\overline{\CL(E^{\,\infty}_{-\infty})}$. 
Then the left-infinite sequences ${\bf a}_{(-\infty,0]} a_1\cdots a_n$ 
defines
$\xi_n\in \BE^0$ so that 
$${\bf a}_{\xi_n}={\bf a}_{(-\infty,0]} a_1\cdots a_n$$ 
for all $n\geq 1$. 
Then obviously 
$\sm({\bf a}_{\xi_{n+1}})={\bf a}_{\xi_n}$, which implies 
together with (\ref{edges}) that 
$e^{a_{n+1}}_{\xi_{n+1}}e^{a_n}_{\xi_n}\in \BE^2$  for $n\geq 1$. 
Thus we obtain an infinite sequence 
$$x:=\cdots e^{a_{3}}_{\xi_{3}}e^{a_2}_{\xi_2}e^{a_{1}}_{\xi_{1}}
\ \text{ in } \BE^\infty.$$

Now it is rather clear from (\ref{xi_n}) that 
the correspondence between 
$x=\cdots e^{a_n}_{\xi_n}\cdots e^{a_1}_{\xi_1}\in \BE^\infty$ and 
${\bf a}_{x}:=\cdots a^3_{\xi_1} a^2_{\xi_1}a_1a_2\cdots a_n\cdots$ in 
$\overline{\CL(E^{\,\infty}_{-\infty})}$ is bijective and continuous. 
Hence it must be a homeomorphism since the spaces 
$\BE^\infty$ and $\overline{\CL(E^{\,\infty}_{-\infty})}$ 
are both compact. 

To show  ${\bf a}_{\sm(x)}=\tau({\bf a}_x)$ for $x\in \BE^\infty$, 
let $x=\cdots e^{a_n}_{\xi_n}\cdots e^{a_2}_{\xi_2}e^{a_1}_{\xi_1}\in \BE^\infty$. 
Then 
$$\sm(x)= \cdots e^{a_3}_{\xi_3} e^{a_2}_{\xi_2}\ \text{ and } \  
{\bf a}_x=\cdots a^3_{\xi_1} a^2_{\xi_1}.a_1a_2\cdots a_n\cdots.$$
Since ${\bf a}_{\xi_2}=\cdots a^3_{\xi_1} a^2_{\xi_1}a_1a_2$, we have 
${\bf a}_{\sm(x)}=\cdots a^3_{\xi_1} a^2_{\xi_1} a_1.a_2\cdots a_n\cdots$.
Therefore 
$${\bf a}_{\sm(x)}
=\tau(\cdots a^3_{\xi_1} a^2_{\xi_1}.a_1a_2\cdots a_n\cdots )
=\tau({\bf a}_{x})$$
 follows.
\end{proof}

\vskip 1pc
The following lemma is immediate from \cite[Theorem 6.4]{Sch:2015}. 
\vskip 1pc 
  
\begin{lem}\label{lem cross product} 
 Let $(\BE^0,\BE^1, d,r)$ be the topological graph of $(E,\CL)$ 
and let $d:\BE^1\to \BE^0$ be injective.    
Then $C^*(E,\CL)$ is isomorphic to the crossed product 
$C\big(\,\overline{\CL(E^{\,\infty}_{-\infty})}\,\big)\times_\tau \mathbb Z$. 
\end{lem} 

\vskip 1pc

\begin{thm} \label{thm simple} 
Let $(\BE^0,\BE^1, d,r)$ be the topological graph of $(E,\CL)$  
and let  $d:\BE^1\to \BE^0$ be injective. 
Assume that $\overline{\CL(E^{\,\infty}_{-\infty})}$ is an infinite set. 
Then we have the following:
\begin{enumerate}
\item[$(a)$] 
$C^*(E,\CL)$ is simple if and only if  
for each ${\bf a}\in \overline{\CL(E^{\,\infty}_{-\infty})}$ 
and $\af\in \CL^*(E)$, there is an $n\in \mathbb Z$ such that 
${\bf a}_{[n,n+|\af|]}=\af$.
 
\item[$(b)$] 
If $C^*(E,\CL)$ is simple, it is always quasidiagonal.
\end{enumerate}
\end{thm}
\begin{proof} 
By Lemma \ref{lem cross product}, it is enought to show that 
every ${\bf a}$ has a dense orbit if and only if it contains 
every possible finite path as its subpath. 
But this easily follows from the fact that 
the cylinder sets $Z(\af.\bt)$ form a basis for the 
topology of $\overline{\CL(E^{\,\infty}_{-\infty})}$. 

Since the dynamical system
$(\,\overline{\CL(E^{\,\infty}_{-\infty})},\, \tau)$ 
is a minimal Cantor system, applying the well known result 
that the crossed product 
$C\big(\,\overline{\CL(E^{\,\infty}_{-\infty})}\,\big)\times_\tau \mathbb Z$ 
is a limit circle algebra 
(for example, see \cite[Theorem VIII.7.5]{Da:1996}), 
we see that the crossed product is a $C^*$-algebra of 
stable rank one (hence stably finite).  
Then Lemma \ref{lem cross product} proves the final assertion of the 
theorem.
\end{proof} 

\vskip 1pc 

\begin{ex} 
Let $(E:=E_{\mathbb Z},\CL)$ be the labeled space of the labeled graph 
considered in Example \ref{ex one}. Then 
$\overline{\CL(E^{\,\infty}_{-\infty})}$ is an infinite set because 
$\tau^n({\bf a})$, $n\in \mathbb Z$, are all distinct 
sequences in $\overline{\CL(E^{\,\infty}_{-\infty})}$, where 
${\bf a}=\cdots 00.1000\cdots$.

On the other hand, the zero sequence ${\bf 0}\in \{0,1\}^{\mathbb Z}$, 
${\bf 0}(n)=0$ for all $n\in \mathbb Z$, which belongs to 
$\overline{\CL(E^{\,\infty}_{-\infty})}$ does not contain any finite 
path $\af\in \CL^*(E)$ whenever $\af$ contains $1$. 
Thus the $C^*$-algebra $C^*(E,\CL)$ is not simple by 
Theorem \ref{thm simple}. 
Actually we always obtain non-simple algebras 
when we label the graph $E_{\mathbb Z}$  
with finite  $1$'s and infinite $0$'s.
\end{ex}

\vskip 1pc 

\begin{cor} \label{cor main}
If $C^*(E,\CL)$ is simple, the following are equivalent: 
\begin{enumerate}
\item[$(a)$] $C^*(E,\CL)$ is AF-embeddable.
\item[$(b)$] $C^*(E,\CL)$ is quasidiagonal.
\item[$(c)$] $C^*(E,\CL)$ is stably finite.
\item[$(d)$] $C^*(E,\CL)$ is finite.
\item[$(e)$] $r(a)\cap r(b)=\emptyset$ for $a,b\in \CA$ with $a\neq b$.
\item[$(f)$] The fixed point algebra $C^*(E,\CL)^\gm$ coincides 
with the diagonal subalgebra 
$\overline{\rm span}\{s_\af p_A s_\af^*\mid \af\in \CL^*(E),\ A\in \CE\}$.
\end{enumerate}
\end{cor}

\vskip 1pc

\begin{ex}
The non-AF but finite simple labeled graph 
$C^*$-algebras $C^*(E_{\mathbb Z},\CL_\om)$ considered in 
\cite[Theorem 3.7]{JKKP:2017} are such examples that satisfy 
condition $(e)$ of Corollary \ref{cor main}.
\end{ex}


\vskip 1pc
 
\begin{thebibliography}{4}

\bibitem[Bates et al.(2015)]{BCP:2015} 
T. Bates, T. M. Carlsen, and D. Pask, 
``$C^*$-algebras of labeled graphs III - $K$-theory computations'', 
\textit{Ergod. Th. $\&$ Dynam. Sys.}, 
\textbf{37} (2017),  337--368.

\bibitem[Bates et al.(2007)]{BP1:2007}  
T. Bates, and D. Pask, 
``$C^*$-algebras of labeled graphs'',
\textit{J. Operator Theory}, \textbf{57} (2007), 101--120.

\bibitem[Bates et al.(2009)]{BP2:2009}  
T. Bates, and D. Pask, 
``$C^*$-algebras of labeled graphs II - Simplicity results'',
\textit{Math. Scand.} \textbf{104} (2009), no. 2, 249--274.    
 
\bibitem[Bates et al.(2012)]{BPW:2012}
T. Bates, D. Pask, and P. Willis,
``Group actions on labeled graphs and their $C^*$-algebras'',
\textit{Illinois J. Math.} \textbf{56} (2012), no. 4, 1149--1168.

\bibitem[Brown(1998)]{Br:1998} N. P. Brown, 
``AF embeddability of crossed products of AF algebras by the integers'',
\textit{J. Funct. Anal.}  \textbf{160} (1998), 150--175.

\bibitem[Brown(2004)]{Br:2004} N. P. Brown, 
``On quasidiagonal $C^*$-algebras'',
\textit{Operator Alagebras and Applications} in Adv. Stud. Pure. Math. \textbf{38}, 
Math. Soc. Japan, Tokyo, 2004, pp. 19--64.

\bibitem[Carlsen et al.(2017)]{COP:2017} 
T. M. Carlsen, E. Ortega, and E. Pardo, 
``$C^*$-algebras associated to Boolean dynamical systems'',
\textit{J. Math. Anal. Appl.} \textbf{450} (2017), 727--768.

\bibitem[Clark et al.(2016)]{CHS:2016} L. O. Clark, A. an Huef, and A. Sims,  
``AF-embeddability of 2-graph algebras and quaidiagonality of $k$-graph algebras'',
\textit{J. Funct. Anal.} \textbf{271} (2016), 958--991.
 
\bibitem[Davidson(1996)]{Da:1996} 
 K. R. Davidson, 
\textit{$C^*$-algebras by Examples},
Fields Institute Monographs, \textbf{6},  Amer. Math. Soc., 1996.
 
\bibitem[Hadwin(2014)]{Ha:1987} 
D. Hadwin, 
``Strongly quasidiagonal $C^*$-algebras'', 
\textit{J. Operator. Th.} \textbf{18} (1987), 3-18.
 
\bibitem[Jeong et al.(2014)]{JKK:2014}  
J. A Jeong,  E. J. Kang and S. H. Kim, 
``AF labeled graph $C^*$-algebras'', 
\textit{J. Funct. Anal.} \textbf{266} (2014), 2153--2173.
 
\bibitem[Jeong et al.(2017)]{JKKP:2017} 
J. A Jeong, E. J. Kang, S. H. Kim and G. H. Park,  
``Finite simple labeled graph $C^*$-algebras of Cantor minimal subshifts'', 
\textit{J. Math. Anal. App.} \textbf{446} (2017), 395--410.

\bibitem[Jeong et al.(2017)]{JKaP:2017} 
J. A Jeong, E. J. Kang,  and G. H. Park,  
``Purely infinite labeled graph $C^*$-algebras'', 
accepted for publication in \textit{Ergod. Th. $\&$ Dynam. Sys.}, 
available on CJO2017. doi:10.1017/etds.2017.123
 
\bibitem[Jeong et al.(2018)]{JP:2018} 
J. A Jeong  and G. H. Park,  
``Simple labeled graph $C^*$-algebras are associated to disagreeable labeled spaces'', 
\textit{J. Math. Anal. Appl.}, 
\textbf{461} (2018), 1391--1403.

\bibitem[Katsura(2004)]{Ka:2004} 
T. Katsura, 
``A class of $C^*$-algebras generalizing both 
graph algebras and homeomorphism $C^*$-algebras I, Fundamental results'', 
\textit{Trans. Amer. Math. Soc.} \textbf{356} (2004), 4287--4322.

\bibitem[Katsura(2006-1)]{Ka:2006-1} 
T. Katsura, 
``A class of $C^*$-algebras generalizing both 
graph algebras and homeomorphism $C^*$-algebras II, Examples'', 
\textit{Internat. J. Math.} \textbf{17} (2006), 791--833.

\bibitem[Katsura(2006-2)]{Ka:2006-2} 
T. Katsura, 
``A class of $C^*$-algebras generalizing both 
graph algebras and homeomorphism $C^*$-algebras III, Ideal structures'', 
\textit{Ergod. Th. $\&$ Dynam. Sys.} \textbf{26} (2006), 1805--1854.

\bibitem[Katsura(2008)]{Ka:2008} 
T. Katsura, 
``A class of $C^*$-algebras generalizing both 
graph algebras and homeomorphism $C^*$-algebras IV, Pure infiniteness'', 
\textit{J. Funct. Anal.} \textbf{254} (2008), 1161--1187.
 
\bibitem[Kitchens(1998)]{Ki:1998} B. P. Kitchens, 
\textit{Symbolic Dynamics}, Springer-Verlag, Berlin Heidelberg 1998. 
 
\bibitem[Pimsner(1983)]{Pim:1983} 
M. V. Pimsner,  
``Embedding some transformation group $C^*$-algebras into AF algebras'', 
\textit{ Ergod. Th. $\&$ Dynam. Sys.}  
\textbf{3} (1983), 613--626.

\bibitem[Schafhauser(2015)]{Sch:2015-1}  C. P. Schafhauser, 
``AF embeddings of graph algebras'',
\textit{J. Operator Th.} \textbf{74} (2015), 177--182.
 
\bibitem[Schafhauser(2015)]{Sch:2015}  C. P. Schafhauser, 
``Finiteness properties of certain topological graph algebras'',
\textit{Bull. London. Math. Soc.} \textbf{47} (2015), 443--454.
 
\bibitem[Tikuisis et al.(2017)]{TWW:2017}  
A. Tikuisis, S. White,  and W. Winter, 
''Quasidiagonality of nuclear $C^*$-algebras'', 
\textit{Ann. of Math.} \textbf{185} (2017), 229--284.
 

\end{thebibliography}
\end{document}